\documentclass{article}[11pt]

\usepackage[a4paper,top=4.0cm,bottom=2.0cm,left=2.0cm,right=2.0cm]{geometry}

\usepackage{amsmath,amsfonts,mathrsfs,amsfonts,ulem}
\usepackage{epsfig,graphicx}
\usepackage{color}
\usepackage{todonotes}
\usepackage{enumerate}
\usepackage{lineno}

\usepackage{amsmath,amsfonts,mathrsfs,amsfonts, color, bbold,bm,blindtext}

\usepackage{algorithmicx,algorithm,algpseudocode}

\allowdisplaybreaks

\newcommand{\E}{\mathbb{E}}
\newcommand{\Prob}{\mathbb{P}}
\newcommand{\calM}{\mathcal{M}}
\newcommand{\calF}{\mathcal{F}}
\newcommand{\dist}{\mathbf{d}} 

\newtheorem{theorem}{Theorem}[section]
\newtheorem{aspt}[theorem]{Assumption}
\newtheorem{defn}[theorem]{Definition}
\newtheorem{cor}[theorem]{Corollary}
\newtheorem{lemma}[theorem]{Lemma}
\newenvironment{proof}[1][Proof]{\begin{trivlist}
\item[\hskip \labelsep {\bfseries #1}]}{\end{trivlist}}

\newenvironment{remark}[1][Remark]{\begin{trivlist}
\item[\hskip \labelsep {\bfseries #1}]}{\end{trivlist}}
\begin{document}

\title{Analysis of a localised nonlinear Ensemble Kalman Bucy Filter with complete and accurate observations}

\author{Jana de Wiljes\thanks{Universit\"at Potsdam, 
			Institut f\"ur Mathematik, Karl-Liebknecht-Str. 24/25, D-14476 Potsdam, Germany ({\tt wiljes@uni-potsdam.de}) and University of Reading, Department of Mathematics and Statistics, Whiteknights, PO Box 220, Reading RG6 6AX, UK} \and Xin T. Tong\thanks{National University Singapore,10 Lower Kent Ridge Road,
Singapore 119076 ({\tt mattxin@nus.edu.sg})
}}
\maketitle

\begin{abstract} 
Concurrent observation technologies have made high-precision real-time data available in large quantities. Data assimilation (DA) is concerned with how to combine this data with physical models to produce accurate predictions. For spatial-temporal models, the Ensemble Kalman Filter with proper localization techniques is  considered to be a state-of-the-art DA methodology.
This article proposes and investigates a localized Ensemble Kalman Bucy Filter (l-EnKBF) for nonlinear models with short-range interactions. We derive dimension-independent and component-wise error bounds and show the long time path-wise error only has logarithmic dependence on the time range. The theoretical results are verified through some simple numerical tests.


\end{abstract}


%

\section{Introduction}
With the advancement of technology, we now have access to vast amounts of high-precision data in many areas of science. It is important to develop robust and efficient tools to combine the available data with refined large-scale physical models. This study is known as data assimilation (DA) and typically the goal is to produce accurate real-time estimations of the current state of the system. 

In geophysical problems, the considered models often have vast spatial scales, therefore millions of state variables are needed to store information at different locations. Such high dimensionality poses a severe challenge to DA methodologies, since the associated computations are expensive and direct global uncertainty quantification tends to be erroneous. Over the last two decades, various computationally feasible approaches have been developed with practical success \cite{saerkkae2013,sr:kalnay,jdw:Evensen2006,jdw:OliverReynoldsLiu2008}. One of the most popular algorithms among these is the Ensemble Kalman filter (EnKF). It has been first derived in \cite{jdw:Evensen2006} and heavily advanced and employed in the field of numerical weather prediction. 
To combat dimensionality issues arising due to the extent of the spatial domain, the so called localization techniques are often employed for the EnKF \cite{sr:reichcotter15,sr:gaspari99}. The key motivation behind localization is that many systems exhibit a natural decrease in spatial correlation. This can guide artificial tunings of the empirical covariance matrix to avoid spurious correlations.


The empirical success of  EnKF has aroused great interest in understanding the underlying theoretical properties \cite{KMT15,sr:majda15,sr:br11,sr:akir11}.  EnKFs can be interpreted as Monte Carlo implementations of the Kalman Filter \cite{jdw:Kalman1960,sr:jazwinski,referee2a,referee2b,referee2c} which is derived for linear prediction and observation models. Therefore most theoretical studies of EnKFs assume a linear setting \cite{mandel2011convergence, BM17, DT18, MT18cpam, Tong18}. Existing analysis of EnKFs for nonlinear models concern mostly the boundedness  of algorithm outputs \cite{sr:KellyEtAl14,sr:majda15,KMT15}, which is not helpful in understanding EnKF performance. The only exception is a recent work \cite{deWiljesReichStannat2018}, where accuracy and stability results have been derived assuming abundant and accurate observations. However, the results there do not consider localization, and hence they require the sample size to be larger than the state dimension. This is infeasible in practice. 

This paper intends to close the aforementioned gaps, i.e. nonlinearity and high dimensionality in filter performance analysis, by investigating a localised Ensemble Kalman-Bucy Filter (l-EnKBF).  Following \cite{deWiljesReichStannat2018}, we  assume abundant and accurate observations are available.
Since most geophysical models are formulated through partial differential equations or their discretizations, the associated prediction dynamics often have a short interaction range. This is often paired with a short decorrelation length in the localization technique to reduce the potential spurious long-range correlations.
Under these assumptions, we show that l-EnKBF  estimation error for \emph{each component is bounded independent of the overall dimension}, both in the sense of mean square and the moment generating function. Such result does not exist in literature for DA analysis, based on our knowledge. Some related dimension-independent error analysis can be found in \cite{MT18cpam,Tong18}, but the error estimates are implicit and the models are assumed to be linear. Moreover, we also show the long time path-wise error has a logarithmic dependence on the time range, which is much weaker than the square root dependence in \cite{deWiljesReichStannat2018}. All these results indicate l-EnKBF has stable and accurate estimation skills.


In Section \ref{sec:setting} the underlying setting is outlined and the considered l-EnKBF will be defined. Upper and lower bounds for the empirical second moment are derived in Section \ref{sec:Stability}. Then point-wise and path-wise bounds for the mean squared error and a Laplace type condition are derived in Section \ref{sec:l2} in the $l_2$ sense, and in Section \ref{sec:comp} in the component-wise sense. We allocate the proofs of our results in the appendix. In Section \ref{sec:Numericalinvestigation}, the numerical sensitivity of an implementation of the considered l-EnKBF with respect to the underlying assumptions is tested for the Lorenz 96 system. 

Throughout the article we assume $(\| \cdot \|, \langle \cdot , \cdot \rangle)$ denotes  the $l_2$-norm with its corresponding inner product. Given a matrix $A \in \mathbb{R}^{m \times n}$ the $l_2$-operator norm is defined as 
\[
\| A \| = \max_{\| x \| = 1} \| Ax \| = \sqrt{\lambda_{\max}(A^{\top} A)},\\
\]
where $\lambda_{\max}$ denotes the largest eigenvalue of a matrix. The following two matrix norms are also useful to us:
\begin{align*}
 &\|A\|_{1}=\max _{1\leq j\leq n}\sum _{i=1}^{m}|A_{i,j}|\\
&\|A\|_{\max} = \max_{ij} |[A]_{i,j}|
\end{align*}
where both $A_{i,j}$ and $[A]_{i,j}$ denote the entries of the matrix $A$. The bracket notation is necessary to denote matrix entries such as $[A^{-1}]_{i,j}$ or $[AB]_{i,j}$. Given two symmetric matrices $A$ and $B$ then $A \succeq B$ implies the matrix $A-B$ is positive semidefinite, which is equivalent to $v^T (A-B) v \geq 0$ for all $v \in \mathbb{R}^n$. Given a covariance matrix $\Gamma$ the Mahalanobis norm  is defined by $\|v\|^2_\Gamma=v^T \Gamma^{-1}v$. Lastly, in order to describe the smallness of certain quantities, we use the Big Theta notation. In particular, a quantity $a_\epsilon$ is $\Theta(\epsilon^p)$, if there is a $\epsilon$-independent constant $C>0$ and $c>0$ so that $c\epsilon^p\leq a_\epsilon\leq C\epsilon^p$.
%
%
%

\section{Problem setup}\label{sec:setting}
In this paper, we consider a continuous-time filtering problem, formulated by 
\begin{equation}
\label{eq:model}
\begin{gathered}
dX_t=f(X_t) dt+\sqrt{2}\sigma dW_t,\\
dY_t=HX_tdt+RdB_t.
\end{gathered}
\end{equation}
In \eqref{eq:model}, $X_t\in \mathbb{R}^{N_x}$ represents the system we try to recover. We assume its initial distribution is given by $X_0\sim \pi_0$.  Its dynamics is driven by a deterministic forcing described by a map $f:\mathbb{R}^{N_x}\rightarrow\mathbb{R}^{N_x}$ and  a stochastic forcing term $\sqrt{2}\sigma dW_t$. We assume linear noisy observations $Y_t\in  \mathbb{R}^{N_y}$ of the system are available. In \eqref{eq:model}, the matrices $\sigma$ and $R$ are positive definite matrices, and $W_t\in \mathbb{R}^{N_x}$ and $B_t\in \mathbb{R}^{N_y}$ are 
independent Wiener processes.

In many spatial models, each model component is representing a state information at one spatial location. This introduces a natural distance between two indices, which we will denote as $\dist$. As a simple example, For example, if the indices are representing themself on the interval $[1,n]$, then $\dist(i,j)$ can be taken as $|i-j|$. For another example, if the indices are representing equally spaced points on a length $n$ circle, then $\dist(i,j)$ be taken as $\min\{|i-j|, n-|i-j|\}$.

We will use $x_i(t)$ to denote the $i$-th component of $X_t$, so $X_t=[x_1(t),\ldots, x_{N_x}(t)]^T$. We will also use $f_i$ and $w_i(t)$ to denote the $i$-th component of $f$ and $W_t$. For notational simplicity, we will often write $x_i(t)$ as $x_i$ and $w_i(t)$ as $w_i$, whenever their dependence on time is evident. Then the SDE that $x_i$ follows is given by 
\begin{equation}
\label{eq:marginal}
dx_i=f_i(X_t)dt+\sqrt{2}\sigma dw_i. 
\end{equation}
Note that different components are interacting through the drift term, as $f_i(X_t)$ could have dependence on $x_j(t)$ for $j\neq i$. But in many physical processes, such interactions are of short range, meaning the dependence of $f_i(X_t)$ on $x_j(t)$ decays with $\dist(i,j)$. More generally, this can be formulated as
\begin{aspt}[Short range interaction]
\label{aspt:short}
There is a sequence of Lipschitz constants $\mathcal{F}_k$, such that for any $X=[x_1,\ldots, x_{N_x}]$ and $X'=[x'_1,\ldots, x'_{N_x}]$, the following holds
\[
|f_i(X)-f_i(X')|\leq \sum_{j=1}^{N_x} \mathcal{F}_{\dist(i,j)} |x_j-x'_j|. 
\]
\end{aspt}
To have a single number controlling the overall stability of the system, we will consider the largest row sum of these Lipschitz constants and define
\begin{equation}
\label{eq:Cf}
 C_f:=\max_i\sum_{j=1}^{N_x} \mathcal{F}_{\dist(i,j)}. 
\end{equation}
 We will assume that $C_f$ is a constant independent of the dimension $N_x$. This can be verified if $\mathcal{F}_k$ decays to zero exponentially with increasing $k$.
 In Section \ref{sec:Numericalinvestigation}, we demonstrate how to verify Assumption \ref{aspt:short} on the Lorenz 96 model, assuming all components are bounded. 
  
In computational models, Assumption \ref{aspt:short} often holds if the spatial resolution is at the same scale of the spatial correlation length. 
 A large $N_x$ indicates that the spatial domain size is large. It is worthwhile mentioning that, 
 it is also possible to obtain a high dimensional model with a moderate size  spatial domain, if one use very small spatial resolution.
 But Assumption \ref{aspt:short} is unlikely to hold in such a setting, and localization techniques are not meant to resolve such high dimensionality. One should use dimension reduction techniques instead \cite{MT18cpam}.
 The difference between these two high dimensional settings are discussed in \cite{MTM19,MTM20}.

\subsection{Localized ensemble Kalman-Bucy filter}
Here we will consider a deterministic EnKBF first proposed in \cite{sr:br11} that has been shown to be the time limit of a broad class of Ensemble Square Root Filters \cite{jdw:LangeStannat2019b}. Let $\{X^i_t\}_{i=1,\ldots,M}$ be the ensemble of particles which describe the uncertainty of $X_t$. To run the considered algorithm, each of the particle is initialized at a random location from $\pi_0$ and then driven by the following dynamics
\begin{equation}
\label{eq:EnKBF}
dX^i_t=f(X^i_t)dt+\sigma^2 P^{-1}_t(X_t^i-\overline{X}_t)dt- \frac{1}{2}P_t H^T(RR^T)^{-1}(HX^i_tdt+H \overline{X}_t dt-2dY(t)).
\end{equation}
In \eqref{eq:EnKBF}, the sample mean and covariance are defined by 
\[
\overline{X}_t=\frac1M \sum_{i=1}^M X^i_t,\quad P_t= \frac1{M-1} \sum_{i=1}^{M} (X^i_t-\overline{X}_t)(X^i_t-\overline{X}_t)^T. 
\]
The posterior distribution of $X_t$ conditioned on $Y_{s\leq t}$ is then approximated by the Gaussian distribution $\mathcal{N}(\overline{X}_t, P_t)$. It is important to mention that for a linear drift $f$ the EnKBF in (\ref{eq:EnKBF}) converges to the KBF for $M\rightarrow \infty$. Further a mean-field limit has been derived for the nonlinear drift scenario \cite{deWiljesReichStannat2018}. Note that mean field limits of EnKFs for a nonlinear setting have also been derived in \cite{referee2d}.  

When the dimension is high, EnKBF is in general ill-defined and it can perform poorly. This is because of two reasons. First, the rank of $P_t$ is at maximum $M-1$. So if $M\ll N_x$, $P_t$ is singular and its numerical approximated inverse is usually unstable. Second, by random matrix theory, it is known that if $X^i_t$ are i.i.d. samples from a Gaussian distribution $\mathcal{N}(0, P)$, in order for the covariance sampler error in $l_2$-norm $\|P-P_t\|$ to be small, one needs  $M=O(N_x)$. 
In other words, $P_t$ is a very inaccurate approximation of the true posterior covariance when $M\ll N_x$ \cite{Tong18}. 

In practice, one popular way to resolve the issues mentioned above is to apply covariance localization. Mathematically, this operation can be formulated as replacing 
$P_t$ in \eqref{eq:EnKBF} with $P^L_t=P_t\circ \phi$. Here $\circ$ denotes the component-wise product or Schur product, so the components of $P_t^L$ are defined as 
\begin{equation}
\label{eqn:PL}
[P^L_t]_{i,j}:=[P_t]_{i,j} \phi_{i,j}.
\end{equation}
The symmetric matrix $\phi$ here is called a localization matrix. Its components are nonnegative. They are of value $1$ at the diagonal, and decay to zero extremely fast along the off diagonal direction. One popular choice takes the form of $\phi_{i,j}=\rho(\frac{\dist(i,j)}{l})$, where $\rho$ is a function from (4.10) in \cite{sr:gaspari99} 
\begin{equation}
\label{eqn:GC}
\rho(x)=\begin{cases}
-\frac14x^5+\frac12x^4+\frac58 x^3-\frac53x^2+1,\quad &|x|\leq 1;\\
\frac1{12}x^5-\frac12x^4+\frac58 x^3+\frac53 x^2-5x+4-\frac2{3x},\quad &1\leq|x|\leq 2;\\
0\quad &2\leq |x|. 
\end{cases}
\end{equation}
where $l$ denotes the typical decorrelation length, which we assume to be independent of $N_x$. We will  consider again the largest row sum of $\phi$, and define
\begin{equation}
\label{eq:Cphi}
 C_\phi:=\max_i\sum_{j=1}^{N_x} \phi_{i,j}. 
\end{equation}
 We will assume $C_\phi$ is a constant independent of the dimension $N_x$. This is true for most practical localization matrices including \eqref{eqn:GC}. 

When the true covariance matrix is spatially localized, $P^L_t$ is a much better covariance estimator, because the localization operation eliminates spurious long distance correlation errors \cite{BL08}. Moreover, the localization operation improves the rank, so $P^L_t$ is often full rank and invertible. But this is not guaranteed in general. So for the rigorousness of this exposition, we use the following inversion
\begin{defn}
\label{defn:DI}
If all diagonal entries of $P_t$ are nonzero, then its diagonal inverse (DI) is given by
\[
[P_t^\dagger]_{i,i}=[P_t]_{i,i}^{-1},\quad [P_t^\dagger]_{i,j}=0,\quad \forall i,j=1,\ldots,n, \,\,i\neq j.
\]
\end{defn}
Note that it satisfies the following for all $i=1,\ldots, n$
\begin{equation}
\label{eqn:DI}
[P_t^\dagger  P_t]_{i,i}=[P_t^\dagger   P_t]_{i,i}=1. 
\end{equation}

In the original EnKBF formulation \eqref{eq:EnKBF}, we replace $P_t$ with $P_t^L$ and $P_t^{-1}$ with $P_t^\dagger$ and we obtain the localized EnKBF (l-EnKBF):
\begin{equation}\label{eq:lEnKBF}
dX^i_t=f(X^i_t)dt+\sigma^2 P^\dagger_t(X_t^i-\overline{X}_t)dt- \frac{1}{2}P^L_t H^T (RR^T)^{-1}(HX^i_tdt+H \overline{X}_t dt-2dY_t).
\end{equation}
As a remark, the using of $P_t^\dagger$ simplifies the theoretical derivation in below, since we can verify that $P_t^\dagger$ is well defined (see Lemma \ref{lem:lowerPt} below). Meanwhile, it is an open question on how to generalize our results to other versions of  pseudo inverse for $P_t$.

\subsection{Abundant and accurate observations}
When the observation sources are abundant, $H$ in \eqref{eq:model} can be assumed to be of rank $N_x$, and there is an $H^-\in \mathbb{R}^{N_x\times N_y}$ such that $H^- H=I_{N_x}$. We can consider the following transformation 
\[
\widetilde{X}_t=\sigma^{-1}X_t,\quad \tilde{f}(X)=\sigma^{-1}f(\sigma X), \quad \widetilde{Y}_t=\sigma^{-1}H^{-} Y_t,\quad \widetilde{R}=\sigma^{-1}H^{-} R,
\] 
then $\widetilde{X}_t$ and $\widetilde{Y}_t$ follow the SDE in below
\begin{equation}
\label{eq:lintran}
\begin{gathered}
d\widetilde{X}_t=\sigma^{-1}dX_t=\tilde{f}(\widetilde{X}_t)dt+\sqrt{2}dW_t,\\
d\widetilde{Y}_t=\sigma^{-1}H^{-} dY_t=\sigma^{-1}X_t+\sigma^{-1}H^{-}R dB_t=\widetilde{X}_tdt+ \widetilde{R} dB_t. 
\end{gathered}
\end{equation}
If we apply l-EnKBF \eqref{eq:lEnKBF} to the transformed system $(\widetilde{X}_t, \widetilde{Y}_t)$, then the sample mean and covariance matrices will follow
\[
\overline {\tilde{x}}_t= \sigma^{-1} \overline{X}_t,\quad \widetilde{P}_t=\sigma^{-2} P_t,\quad \widetilde{P}^L_t=\sigma^{-2} P^L_t,
\]
while $\widetilde{P}^\dagger_t$ can be taken as $\sigma^{2} P^\dagger_t$. Then the dynamics of each l-EnKBF particle will satisfy
\begin{align*}
d\widetilde{X}^i_t&=\widetilde{f}(\widetilde{X}^i_t)dt+\widetilde{P}^\dagger_t(\widetilde{X}_t^i-\overline {\tilde{x}}_t)dt- \frac{1}{2}\widetilde{P}^L_t (\widetilde{R} \widetilde{R}^T)^{-1}(\widetilde{X}^i_tdt+\overline{\tilde{x}}_t dt-2d\widetilde{Y}_t)\\
&=\sigma^{-1}f(X^i_t)dt+\sigma P^\dagger_t(X_t^i-\overline{X}_t)dt- \frac{1}{2}\sigma^{-1}P^L_t H^T(RR^T)^{-1}(HX^i_tdt+H \overline{X}_t dt-2dY_t)=\sigma^{-1}dX^i_t. 
\end{align*}
It is evident that the theoretical properties of $X^i_t$ will be the same as the ones of $\widetilde{X}^i_t$.

Note that \eqref{eq:lintran} corresponds to the original model \eqref{eq:model} with $\sigma=1$ and $H=I$. This is a much simplified parameter setting for followup discussion. And from the above derivation, there is no sacrifice of generality by focusing on it. Under this setting, the l-EnKBF formula will be simplified as 
\[
dX^i_t=f(X^i_t)dt+ P^\dagger_t(X_t^i-\overline{X}_t)dt- \frac{1}{2}P^L_t \Omega_R (X^i_tdt+\overline{X}_t dt-2dY_t),\quad \Omega_R:=(RR^T)^{-1}. 
\]
When the observations are accurate and independent, the observation noise covariance $RR^T$ is a diagonal matrix with small components. We will use $\epsilon$ to describe their order. In summary, we have made the following assumption 
\begin{aspt}
\label{aspt:trans}
Through a linear transformation, we assume \eqref{eq:model} is transformed to
\[
\begin{gathered}
dX_t=f(X_t) dt+\sqrt{2} dW_t,\\
dY_t=X_tdt+R dB_t
\end{gathered}
\]
Moreover we assume for an $\epsilon>0$ that $\Omega=\epsilon (RR^T)^{-1}$ is diagonal, and bounded by constants $\omega_{\min }I\preceq \Omega\preceq \omega_{\max}I$. 
\end{aspt}
Note that Assumption \ref{aspt:trans} implies that $RR^T=\Theta(\epsilon)$. In other words we assume that the squared observation error covariance matrix is of order $\epsilon$.

By replacing $\Omega_R$ with $\epsilon^{-1}\Omega$, the l-EnKBF formula is written as
\begin{equation}\label{sys:lEnKBF}
dX^i_t=f(X^i_t)dt+ P^\dagger_t(X_t^i-\overline{X}_t)dt- \frac{1}{2\epsilon}P^L_t \Omega (X^i_tdt+\overline{X}_t dt-2dY_t). 
\end{equation}
Since $\overline{X}_t=\frac1M \sum_{i=1}^MX^i_t$, the sample mean process follows the following dynamics
\begin{equation}
\label{eq:evolutionmean}
d\overline{X}_t=\overline f_tdt- \epsilon^{-1}P^L_t \Omega ( \overline{X}_t dt-dY_t),\quad \overline f_t:=\frac1M\sum_{i=1}^M f(X^i_t). 
\end{equation}
So if we denote $\Delta X^i_t=X^i_t-\overline{X}_t$, it follows the ordinary differential equation (ODE)
\[
\frac{d}{dt}\Delta X^i_t=f(X^i_t)-\overline f_t+P^\dagger_t\Delta X^i_t-\frac{1}{2\epsilon}P^L_t \Omega \Delta X^i_t.
\]
Because the sample covariance $P_t=\frac1{M-1}\sum_{i=1}^{M} \Delta X^i_t (\Delta X^i_t)^T$, we have
\begin{equation}
\label{eq:evolutionP_t}
\frac{d}{dt} P_t= (F_t+F_t^T)+(P_t^\dagger P_t+P_t P_t^\dagger ) -\frac1{2\epsilon}(P^L_t \Omega P_t+P_t \Omega P^L_t)
\end{equation}
where $F_t:=\frac1{M-1} \sum (X^i_t-\overline{X}_t)(f(X^i_t)-\overline f_t)^T$.

%

\section{Main results}
We present  our main theoretical results for the l-EnKBF in \eqref{eq:lEnKBF} in this section. To keep to discussion  concise, we allocate the technical verifications to the appendix.

\subsection{Wellposedness and Stability}\label{sec:Stability}
Before the accuracy of the filter can be addressed it is crucial to check if the l-EnKBF can blow-up or collapse. In other words, we will demonstrate that the filter is stable, such that there are upper and lower bounds for $P_t$. The upper bound is established by the following:

\begin{lemma}
\label{lem:upperPt}
Under Assumptions \ref{aspt:short} and \ref{aspt:trans}, suppose $P_t^L$ evolving in time according to (\ref{eq:evolutionP_t}) exists, the following holds
\[
\|P_t\|_{\max}\leq  \lambda_{\max}:= \frac{2\epsilon}{\omega_{\min}}\left(\sqrt{C_f^2+\frac{3\omega_{\min}}{\epsilon} }\right),\quad \forall t>t'_*:=\frac{\omega_{\min} \epsilon}{\lambda_{\max}}.
\]
And for all $t>0$, $\|P_t\|_{\max}\leq \max\{\|P_0\|_{\max}, \lambda_{\max}\}$. It is clear that when $C_f$ and $\omega_{\min}$ are constants, $\lambda_{\max}(\epsilon)=\Theta(\sqrt{\epsilon}),t'_*=\Theta(\sqrt{\epsilon})$. 
\end{lemma}
In \cite{deWiljesReichStannat2018} the bound depends explicitly on $M$ (as the Frobenius norm is used to derive the bound). Here a different route is taken which results in a bound independent of $M$.

To ensure that the filter does not collapse, it is crucial to have a lower bound on the covariance. This comes as a reverse of Lemma \ref{lem:upperPt}. For this purpose, we denote 
\[
\|P_t\|_{\min}=\min\{[P_t]_{i,i},i=1,\ldots,N_x\}. 
\]
It should be noted that $\|P_t\|_{\min}$ is not a norm, and we choose this notation just for its symmetry with $\|P_t\|_{\max}$. 
\begin{lemma}
\label{lem:lowerPt}
Under Assumptions \ref{aspt:short} and \ref{aspt:trans}, suppose $P_t^L$ evolves in time according to (\ref{eq:evolutionP_t}),  the following holds for sufficiently small $\epsilon>0$
\[
\|P_t\|_{\min}\geq  \lambda_{\min}:= \frac{\epsilon}{3\lambda_{\max}\omega_{\max} C_\phi},\quad \forall t>t_*:=\frac{\omega_{\min} \epsilon}{\lambda_{\max}}+3\lambda_{\min}.
\]
\[
\|P_t\|_{\min}\geq \min\left\{\|P_0\|_{\min},\frac{\epsilon}{2\omega_{\min}\max\{\|P_0\|_{\max}, \lambda_{\max}\} } \right\} >0,\quad\forall t>0.
\]
It is clear that when $C_f$ and $\omega_{\min}$ are constants, $\lambda_{\min}(\epsilon)=\Theta(\sqrt{\epsilon}),t_*=\Theta(\sqrt{\epsilon})$. 
\end{lemma}

Since $P_t^\dagger$ is well defined as long as $\|P_t\|_{\min}>0$,  using the same proof as in Theorem 2.3 of \cite{deWiljesReichStannat2018}, we can show that the l-EnKBF given by \eqref{eq:lEnKBF} has a strong solution:
\begin{cor}
\label{cor:exists}
Suppose the initial ensemble is selected so that $\|P_0\|_{\min}>0$. Then the l-EnKBF filter is well defined for all $t>0$. 
\end{cor}



\subsection{Error analysis in $l_2$ norm}
\label{sec:l2}
As the next step we consider the accuracy of l-EnKBF in terms of the  $l_2$ norm. Since the filter estimate with the ensemble mean, the error is its deviation from the truth, $e_t=X_t-\overline{X}_t$. While it has already been shown in \cite{deWiljesReichStannat2018} $\|e_t\|^2$ is of order $N_x\sqrt{\epsilon}$ through tail probability, our new result extends this estimate to the Laplace transforms. Moreover we show the path-wise maximum has the logarithm scaling with time, indicating the filter is highly stable in terms of error. 
\begin{theorem}
\label{theo:accuracyl2} 
Let $ e_t=X_t - \overline{X}_t$ be the filter error of l-EnKBF \eqref{sys:lEnKBF}. Under Assumptions \ref{aspt:short} and \ref{aspt:trans}, if $\tilde{\phi}:=\phi-\rho I\succeq \mathbf{0}$ for a constant $\rho>0$, then  for any fixed $t_0>0$ there are strictly positive constants $\epsilon_0,c$ and $C$ such that for every $\epsilon\in (0,\epsilon_0)$, 
\begin{enumerate}[1)]
\item When $t>t_0, \E\|e_t\|^2\leq  C\sqrt{\epsilon} N_x $.
\item For any $0<\lambda<c\epsilon^{-1/2}$,
\[
\limsup_{t\to \infty}\E \exp(\lambda\|e_t\|^2)\leq 2\exp(4C\lambda N_x \sqrt{\epsilon}) .
\]
\item For any $T>t_0$, the following holds
\[
\E_{t_0}\left[\sup_{t_0\leq t\le T}\|e_t\|^2\right]\leq \|e_{t_0}\|^2+ C \sqrt{\epsilon} N_x+C\sqrt{\epsilon} \log (CT/\sqrt{\epsilon}) 
\]
Here $\E_{t_0}$ denotes conditional expectation with respect to information available at time $t_0$. 
\end{enumerate}
\end{theorem}
Note that the  $\epsilon^{1/2}$ scaling is sharp. This can be understood best if one applies the Kalman-Bucy filter to \eqref{eq:model} with $f(X)=\mathbf{0}$, $H=I_{N_x}$ and $R=\sqrt{\epsilon} I_{N_x}$, the posterior covariance $P_t$ follows the ODE $\frac{d}{dt} P_t=2I_{N_x}-\epsilon^{-1}P_t^2$. It is easy to show that $P_t$ will converge to the limit $P_\infty=\sqrt{2\epsilon}I_{N_x}$, which is of order $\epsilon^{1/2}$ as well.

\subsection{Analysis for component-wise error}
\label{sec:comp}
While Theorem \ref{theo:accuracyl2} provides an estimate $\|e_t\|^2$, the estimate has a scaling of $N_x$ because $\|e_t\|^2$ is the sum of $N_x$ component errors. From Theorem \ref{theo:accuracyl2}, it is impossible to indicate the error of one specific component, or whether this component's error is independent of the dimension $N_x$. This section shows that with a stronger structure assumption on the localization matrix, we can derive dimension-independent bounds for each individual component. 
\begin{aspt}
\label{aspt:diagd} The localization matrix $\phi$ is diagonally dominant. In other words, there is a $q<1$ such that 
\[
\sum_{j\neq i}\phi_{i,j}\leq q. 
\]
Moreover, the interaction between components can be dominated by a constant $C_{\mathcal{F}}$-multiple of the matrix structure $\phi$:
\[
\mathcal{F}_{\dist(i,j)}\leq C_{\mathcal{F}} \phi_{i,j}\quad \forall i,j. 
\]
\end{aspt}

Since $\mathcal{F}_{d}$ usually decays to zero quickly in practice, so $C_{\mathcal{F}}$ are likely to be found.  Using Lemma \ref{lem:norm} it is easy to show $\phi$ satisfying Assumption \ref{aspt:diagd} will have $\phi\succeq (1-q) I$, meaning $\tilde{\phi}=\phi-qI$ is positive semidefinite. In other words, Assumption \ref{aspt:diagd} is stronger than assumption for $\phi$ imposed in Theorem \ref{theo:accuracyl2}. In general,  $\phi$ is not always  diagonally domain. However, this can hold if one choose small localization length $l$. For example, for the  Gaspari--Cohn \cite{sr:gaspari99} distance matrix $\phi$, it will be diagonally dominant if $l\leq 1.4$. In other words, Assumption \ref{aspt:diagd} is likely to hold if the components of model represent spatial information of distant apart. 

With Assumption \ref{aspt:diagd}, we can reproduce Theorem \ref{theo:accuracyl2} type of result for individual component. 
\begin{theorem}
\label{the:individualcomp}
Let $ e_t=X_t - \overline{X}_t$ be the filter error of l-EnKBF \eqref{sys:lEnKBF}. Under Assumptions \ref{aspt:short}, \ref{aspt:trans}, and \ref{aspt:diagd}, for any fixed $t_0>0$ there are  constants $c$ and $C$ such that for sufficiently small $\epsilon>0$, 
\begin{enumerate}[1)]
\item When $t>t_0$, for any index $i$, $\E [[e_t]_i^2]\leq  C\sqrt{\epsilon}$.
\item For any $0<\lambda<c\epsilon^{-1/2}$ and index $i$,
\[
\limsup_{t\to \infty}\E \exp(\lambda [e_t]_i^2)\leq 2\exp(4C\lambda \sqrt{\epsilon}) .
\]
\item For any $T>t_0$, the following holds for all $i$
\[
\E_{t_0}\left[\sup_{t_0\leq t\le T}[e_t]^2_i\right]\leq \max_i\{[e_{t_0}]^2_i\}+ C\sqrt{\epsilon} \log (T/\sqrt{\epsilon}). 
\]
Here $\E_{t_0}$ denotes conditional expectation with respect to information available at time $t_0$. 
\item For any $T>t_0$,
\[
\E_{t_0}\left[\max_i\sup_{t_0\leq t\le T}[e_t]^2_i\right]\leq \max_i\{[e_{t_0}]^2_i\}+ C\sqrt{\epsilon} \log (N_xT/\sqrt{\epsilon}). 
\]
\end{enumerate}
\end{theorem}
\begin{remark}
If $Z_1,\ldots,Z_n$ are i.i.d. samples of a Gaussian distribution, a rough estimate of $\max_i\{Z_i\}$ is of order $\log n$. And when system has short range interaction, its components are tend to be independent when they are far apart. Likewise, when a system is stationary, it is close to independent with it self in a distance past. The filter error process happens to have both of these two properties. That is why we have the scaling of $\log(N_xT)$ in claim 4). 
\end{remark}

\section{Numerical investigation}\label{sec:Numericalinvestigation}
Lastly the theoretical findings are numerically verified by means of the stochastically perturbed Lorenz 96 system (L96) \cite{sr:lorenz96}. The evolution of each spatial component is given by 
\[
d x_{s}(t)=f_s(X(t))dt+\sqrt{2} dW_s(t)
\] 
for $s\in \{1,\dots,N_x\}$. Here
\begin{equation}\label{eq.L96}
f_s(X(t))=\Big({x}_{s+1}(t)-{ x}_{s-2}(t)\Big){x}_{s-1}(t)-{x}_s(t)+8,
\end{equation}
and spatial periodicity is assumed, i.e., $x_{-1}(t)=x_{N_x-1}(t)$, $x_0(t)=x_{N}(t)$ and $x_{N_x+1}(t)=x_{1}(t)$. Numerically generated trajectories of \eqref{eq.L96} are typically bounded in the $l_\infty$ norm, i.e.,
\begin{equation}
|x_s(t)| \le C=40
\end{equation}
for all $s$ for the Lorenz 96 system. In other words, the solution of \eqref{eq.L96} is largely indifferent from a soft-truncated version $d X_{s}(t)=\hat{f}_s(X(t))dt+\sqrt{2} dW_s(t) $, where
\begin{equation}\label{eq.L96trunc}
\tilde{f}(x_{s}(t))=1_{\|X(t)\|_\infty\leq C}\Big({x}_{s+1}(t)-{x}_{s-2}(t)\Big){ x}_{s-1}(t)-{x}_s(t).
\end{equation}
 Then note that when $\|X(t)\|_\infty\geq C$, $|\tilde{f}_s(X(t))-\tilde{f}_s(X'(t))|=|{x}'_s(t)-{x}_s(t)|$; when $\|X(t)\|_\infty\leq C$,
\begin{align*}
|\tilde{f}_s(X(t))-\tilde{f}_s(X'(t))|&=|\Big({x}_{s+1}(t)-{ x}_{s-2}(t)\Big){ x}_{s-1}(t)-{x}_s(t)-[\Big({x}'_{s+1}(t)-{ x}'_{s-2}(t)\Big){ x}'_{s-1}(t)-{x}'_s(t)]|\\
&\le|({ x}_{s+1}(t)-{ x}'_{s+1}(t)){ x}_{s-1}(t)|+|({ x}_{s-1}(t)-{ x}'_{s-1}(t)){ x}'_{s+1}(t)|\\
&+|({ x}'_{s-2}(t)-{ x}_{s-2}(t)){ x}'_{s-1}(t)|+|({ x}'_{s-1}(t)-{ x}_{s-1}(t)){ x}_{s-2}(t)| + |{x}'_s(t)-{x}_s(t)|\\
&\le C|{ x}_{s+1}(t)-{ x}'_{s+1}(t)|+C|{ x}_{s-1}(t)-{ x}'_{s-1}(t)|+C|{ x}'_{s-2}(t)-{ x}_{s-2}(t)|\\
&+C|{ x}'_{s-1}(t)-{ x}_{s-1}(t)|+|{x}'_s(t)-{x}_s(t)|.
\end{align*}
Therefore, Assumption \ref{aspt:short} is fulfilled with $\mathcal{F}_{\dist(i,j)}$=0 for $\dist(i,j)>2$ where $\dist(i,j)=\min\{|i-j|, |i+n-j|, |j+n-i|\}$, and (\ref{eq.L96trunc}) has only short range interactions. While we will only simulate \eqref{eq.L96} in below, we expect the associated filter behavior will be similar to the one in \eqref{eq.L96trunc}. Further the entries of the localization matrix $\phi$ are set to 
\begin{equation*}
\phi_{i,j}=\rho\Big(\frac{\dist(i,j)}{l}\Big)
\end{equation*}
using the Gaspari--Cohn function (\ref{eqn:PL}) for $\rho$ and setting the localization radius to $l=1.4$. Note that this choice of localization radius ensures that $\phi$ is diagonally dominant, i.e., Assumption \ref{aspt:diagd} is fulfilled. It is important to note that this choice is not necessarily the \textit{optimal}\footnote{Here optimality can for example be associated with the lowest MSE.} value for the considered system yet the chosen value is sufficient to obtain reasonable MSE values of the expected order. Further we choose the model noise variance to be $\sigma=1$ and the observation operator $H$ to be the identity matrix which is in line with Assumption \ref{aspt:trans}. Three test scenarios are considered to numerically verify the sensitivity of the l-EnKBF with respect to the dimension $N_x$, time interval size $T$ and the measurement error $\epsilon$.
\subsection{Sensitivity with respect to $\epsilon$}
In the first test scheme, the expected filtering error is approximated via a time-averaged MSE for different measurement error values 
\begin{equation*}
\epsilon \in \{0.003125, 0.00625, 0.025, 0.05, 0.1\}.
\end{equation*}
In order to emulate a continuous setting the steps size is chosen to be $dt=10^{-7}$ and the number of steps $10^7$. The dimension of the state space is set to be $N_x=40$, which is a standard choice of the Lorenz 96 model. The l-EnKBF is implemented with $M=10$ ensemble members.
The results are displayed in the left panel of Figure \ref{fig:increasingepsandT}. Note that the MSE is normalised with respect to the dimension, i.e., is divided by $N_x$. The test run confirms that the numerical growth rate with respect to an increasing $\epsilon$ is in line with theoretical order of the expected error derived in claim $1)$ of Theorem \ref{theo:accuracyl2}.
\begin{figure}\label{fig:increasingepsandT}
\begin{center}
\includegraphics[width=0.45\textwidth]{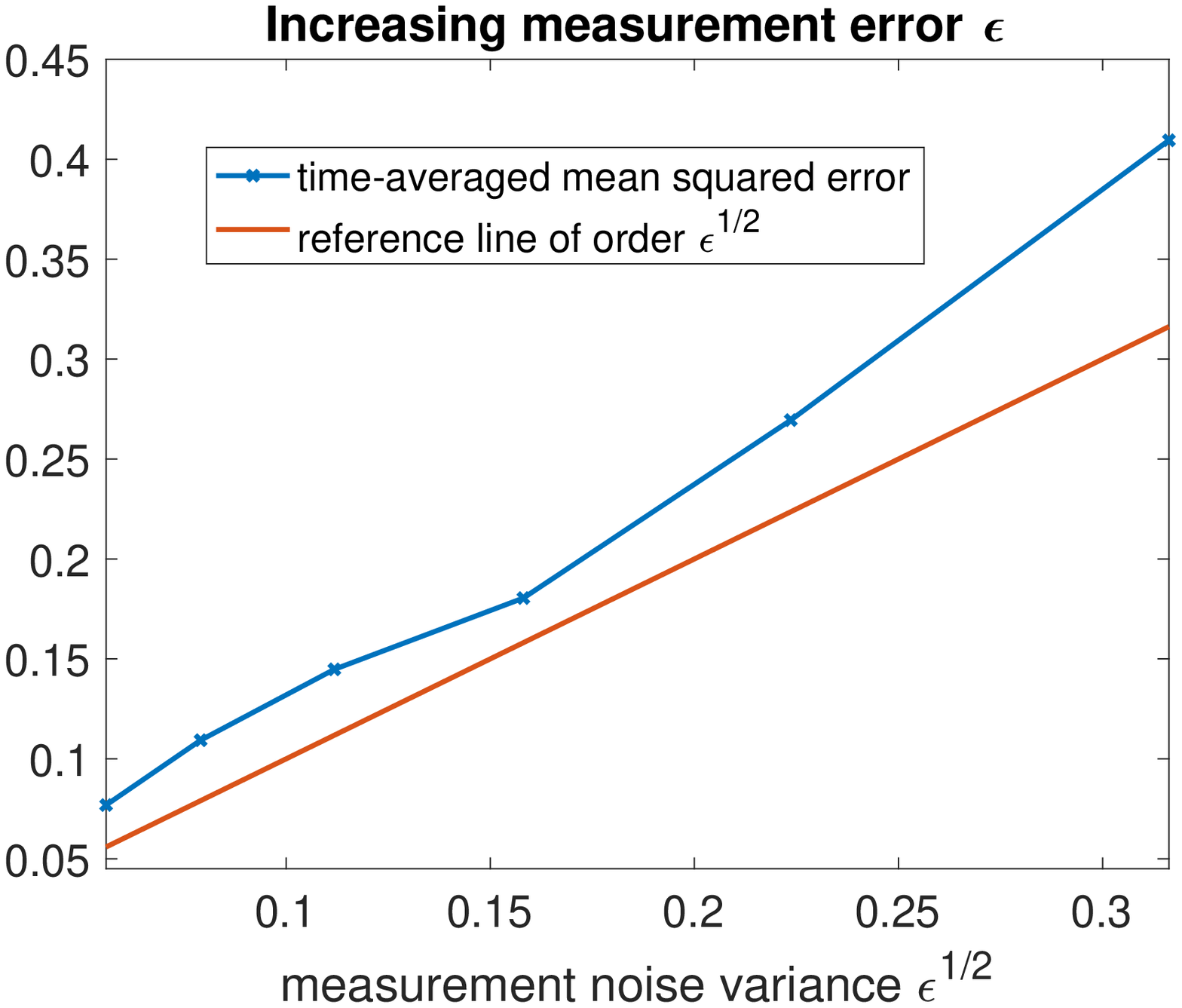} 
\includegraphics[width=0.45\textwidth]{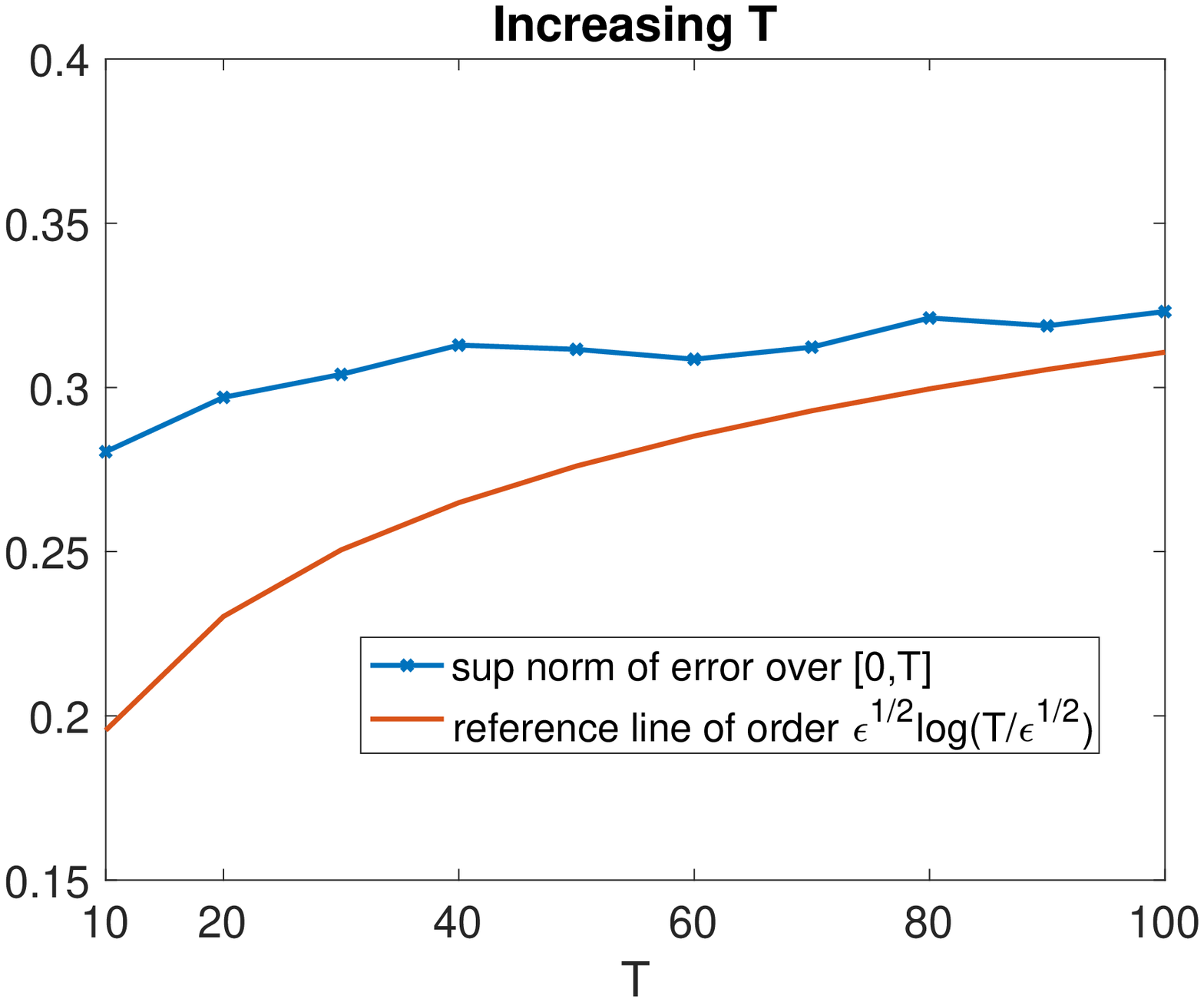} $\quad$
\end{center}
\caption{Time-averaged MSE as a function of the measurement error variance $\epsilon$ is displayed in the left panel. The right panel displays the estimated $\sup_{t\in[0,T]}e_t$ for varying $T$.}
\end{figure}
\subsection{High dimensional testcase}
In the second test scheme, the robustness with respect to state space dimension is investigated. In particular we consider the case where the number of ensemble members $M$ is comparatively small and kept fixed for increasing dimensions. Thus the imbalance between ensemble size and dimension of the state space grows with increasing $N_x$. More precisely we run the filter for $N_x\in\{40, 240, 440,640,840,1040\}$ with $M=10$ and $\epsilon=0.003125$. The resulting time-averaged MSE after $10^6$ steps with step size $dt=10^{-7}$ are displayed in the left panel of Figure \ref{fig:increasingdim}. As state in claim $1)$ of Theorem \ref{theo:accuracyl2} the error grows linearly with $N_x$. Further we numerically verify that the time-averaged error of the individual components, i.e., 
\begin{equation}
\frac{1}{T}\sum^T_{t=1} [e_t]_i^2(t)
\end{equation}
are dimension independent (see right panel of Figure \ref{fig:increasingdim}) as stated in claim $1)$ Theorem \ref{the:individualcomp}. Note that we fixed the considered component of the state vector to be $i=11$ while other index choice produces largely the same results.
\begin{figure}\label{fig:increasingdim}
\begin{center}
\includegraphics[width=0.45\textwidth]{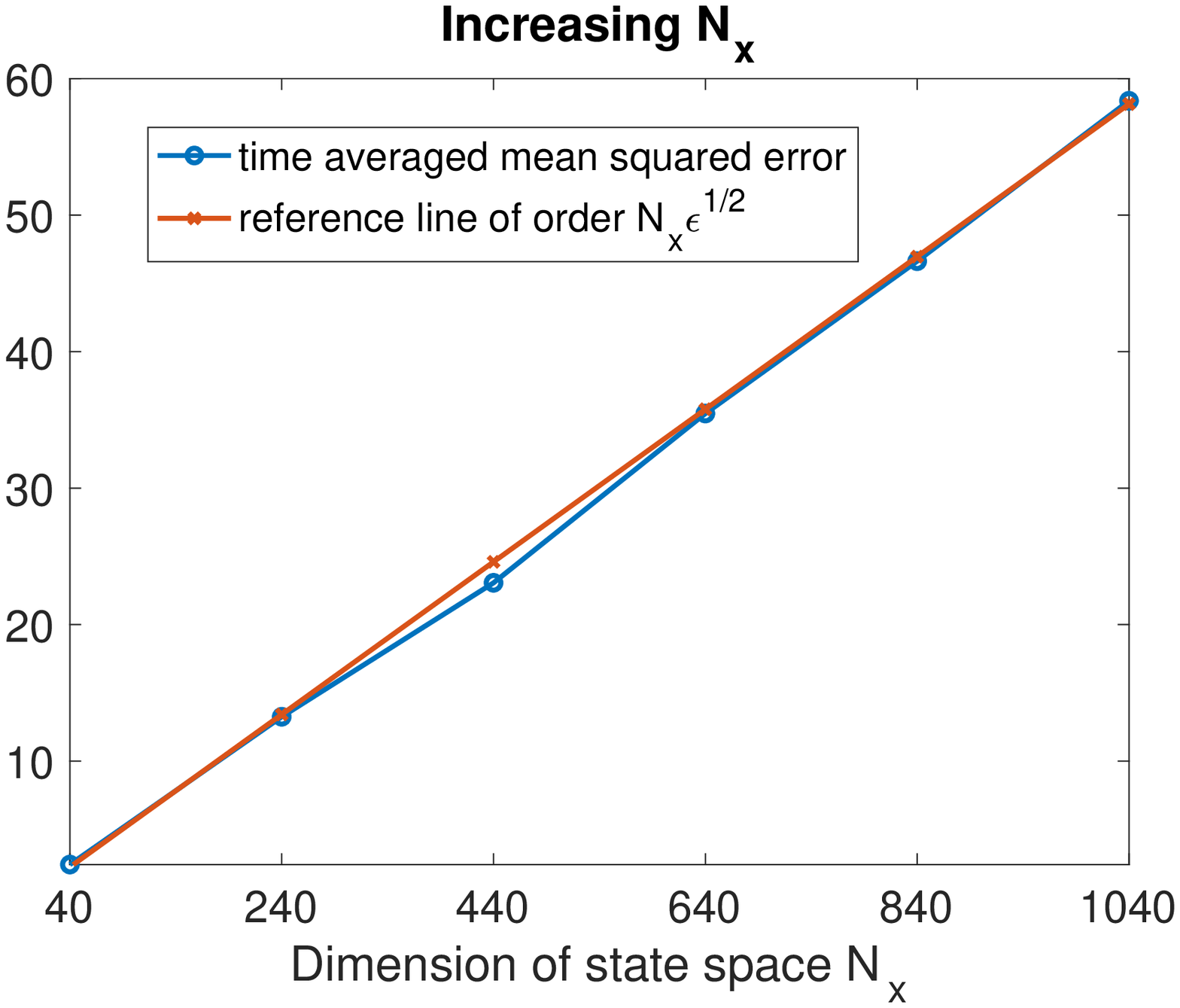} 
\includegraphics[width=0.45\textwidth]{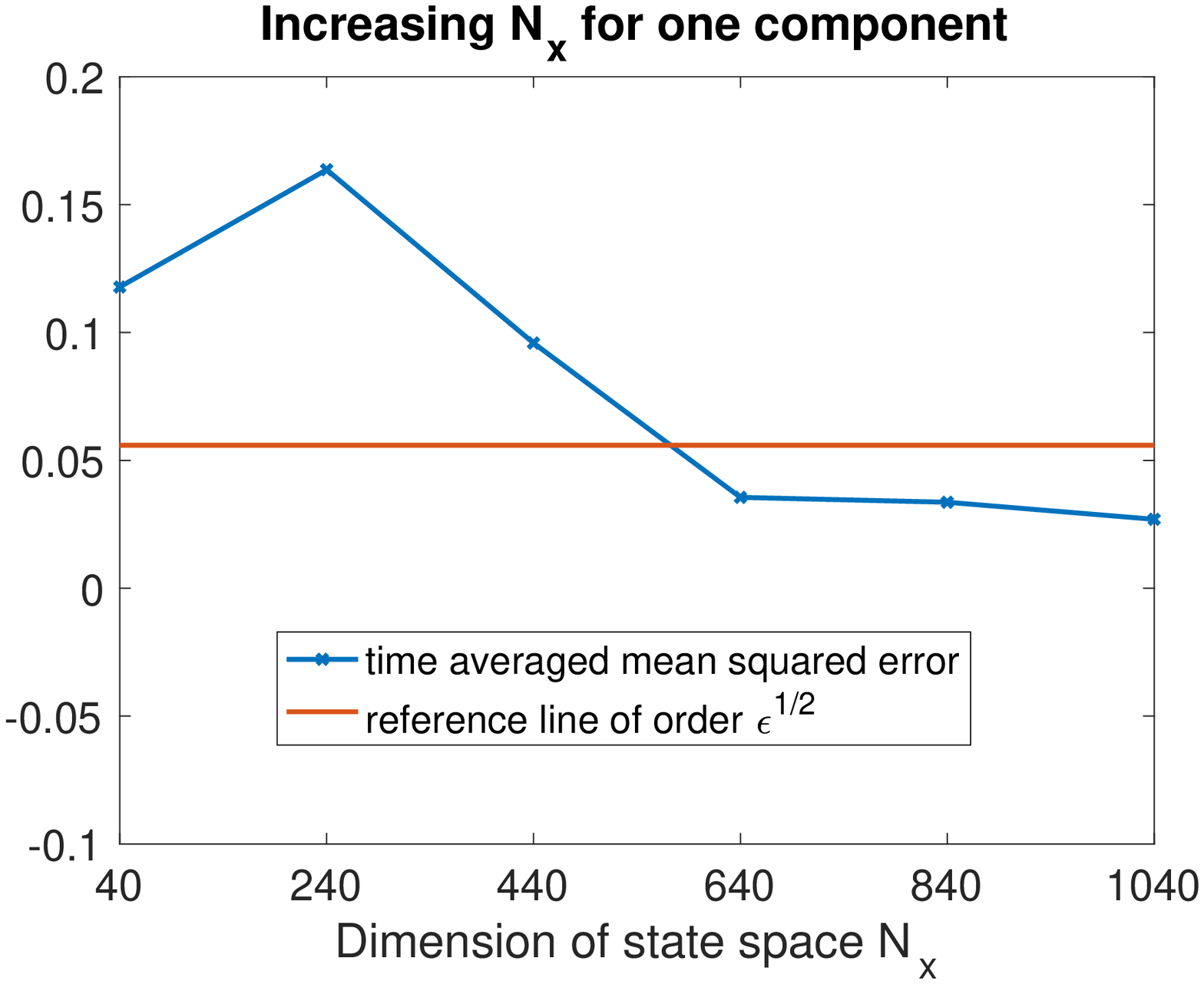}$\quad$
\end{center}
\caption{Time-averaged MSE as a function of the state dimension $N_x$ is displayed in the left panel whereas the time-averaged MSE for one fixed component for increasing $N_x$ is shown in the right panel.}
\end{figure}

\subsection{Uniform error for bounded time interval}
In the final test scheme, we consider a setting with a growing number of steps $10^{6}$ to $10^{7}$ for a fixed step size $dt=10^{-5}$ resulting in filter runs for different time values $T\in\{10,\dots,100\}$. Note that the step size is set to be slightly larger than in the previous examples so that the range of considered $T$ values is more interesting. Further the measurement error variance is set to $\epsilon=0.01$ and the dimension of the state space is $N_x=40$. We simulate the filtering process $30$ times and record the filter error $e^j(t), j=1,\ldots, 30,$ for each simulation. We plot the averaged path-wise $l_2$-square error up to $T$, which is 
\begin{equation*}
\frac{1}{30}\sum^{30}_{j=1} \max_{t\in[0,T]} ||e^j(t)||^2
\end{equation*} 

in the right panel of Figure \ref{fig:increasingepsandT}. The dominating part\footnote{For the considered $N_x$, $\epsilon$ and $T$, the other dominating component $C\sqrt{\epsilon}N_x$ is not large but can of course become significant for $N_x>>0$.} $C\sqrt{\epsilon} \log (CT/\sqrt{\epsilon})$ of the theoretical order of claim $3)$ of Theorem \ref{theo:accuracyl2} is plotted as a reference slope. 

Note that the numerically obtain error is in line with the theoretical order and thus is verifying the logarithmic dependence of the uniform bound on time $T$.

\section{Conclusion}
%

In this paper, the earlier derived stability and accuracy results for the EnKBF are extended for systems with $N_x>>M$ via localization. Further the upper bound for the covariance is independent of the number of ensemble members $M$ and the derived path-wise bounds have a better scaling with respect to the time T. Moreover it is shown that the accuracy in the individual components is independent of the state dimension $N_x$ and a Laplace type condition is obtained. Natural extensions include partially observed processes and misspecified drift functions $f(x_t,\lambda)$ with unknown parameter $\lambda$.
Moreover the presented ideas can be used for the analysis of properties of Multilevel Ensemble Kalman Filters \cite{referee2e,referee2f} or of consistent filters, such as the Ensemble Transform Particle Filter \cite{sr:reich13} or the Feedback Particle Filter \cite{sr:meyn13}, for finite number of ensemble members.

\section*{Acknowledgement}
%

The research of J.d.W is funded by Deutsche Forschungsgemeinschaft (DFG) - SFB1294/1 - 318763901 Project (A02) \lq\lq 
Long-time stability and accuracy of ensemble transform filter algorithms\rq\rq. JdW was also supported by ERC Advanced Grant ÒACRCCÓ (grant 339390) and by the Simons CRM Scholar-in-Residence Program. The research of X.T.T is funded by Singapore MOE AcRF Tier 1 funding, grant number R-146-000-292-114. Further the authors would like to thank the anonymous referees for their constructive criticism.

\appendix 

\section{Proof for filter wellposedness and stability}
\subsection{Matrix norms and Riccati equation}
To start, we have several norm inequalities which are utilized in this paper.
\begin{lemma}
\label{lem:norm}
For any $N\times N$ matrix $A$, the following holds
\begin{align}
&\|A\|_{\max} \le \|A\|,\label{eq:maxle2norm}\\
&\|A\| \le \sqrt{\|A\|_1 \|A^T\|_1}\label{eq:2normlestuff}.
\end{align}
\end{lemma}
\begin{proof}
Inequality (\ref{eq:maxle2norm}) follows via
\[
\|A\|_{\max}=\max_{i,j}|[A]_{i,j}|=\max_{i,j}|[e_t]_i^TA e_j|\le \|A\|,
\]
where $[e_t]_i$ and $e_j$ are the $i$-th and $j$-th standard Euclidean basis vector. Inequality \eqref{eq:2normlestuff} follows from \cite{MTM19} Lemma B.2.
\end{proof}

\begin{lemma}\label{lem:aux1}
Let  $P$, $Q$ and $\phi$ be positive, symmetric and semidefinite $N_x\times N_x$ matrices and $[\phi]_{i,i}=1$ for all $i$. Then 
\begin{enumerate}[1)]
\item For all $i$, $[(P\circ \phi) Q]_{i,i}=[P(Q\circ \phi)]_{i,i}$.
\item If $P\preceq Q$, then $P\circ \phi\preceq Q\circ\phi$. 
\item  $\|P\circ \phi\|_{\max}=\|P\|_{\max}=\max_{i}\{[P]_{i,i}\}$
\item $\|P \circ \phi\|\le \|P\circ \phi\|_1\leq  C_\phi \|P\|_{\max}$, where $  C_\phi =\max_i\sum_j |\phi_{i,j}|$.
\end{enumerate}

\end{lemma}
\begin{proof}[Proof Claim 1)]
Just note that
\[
[(P\circ \phi) Q]_{i,i}=\sum_k P_{i,k}\phi_{i,k}Q_{k,i}=\sum_k P_{i,k}\phi_{k,i}Q_{k,i}= [P (Q\circ \phi)]_{i,i}. 
\]
\end{proof}
\begin{proof}[Proof Claim 2)]
Due to the linearity of the Schur product, it suffices to show that $0\preceq P\circ \phi$. This is known as the Schur product theorem, which can be verified using the following identity, which holds for all $N_x$-dimensional vectors $u$, with $D_u$ being the diagonal matrix where its diagonal entries are the same as $u$:
\[
u^T(P\circ \phi) u=\text{tr}(PD_u \phi D_u)=\text{tr}(P^{\frac12}D_u \phi^{\frac12} \phi^{\frac12} D_u P^{\frac12})\geq 0.
\]
\end{proof}
\begin{proof}[Proof Claim 3)]
Since $P$ is a positive semidefinite matrix for each $i$ and $j$, it follows that
\begin{equation*}
 \langle [e_t]_i-e_j,P([e_t]_i-e_j)  \rangle=\langle [e_t]_i, P [e_t]_i\rangle -2 \langle [e_t]_i, P e_j\rangle+\langle e_j, P e_j\rangle\ge 0,
\end{equation*}
where $[e_t]_i$ and $e_j$ are the $i$ and $j$-th standard Euclidean basis vector. This implies 
\begin{equation*}
2 \langle [e_t]_i, P e_j\rangle \le \langle [e_t]_i, P [e_t]_i\rangle +\langle e_j, P e_j\rangle\le 2\max_{k} [P]_{k,k}.
\end{equation*}
In other words in a positive semidefinite matrix the maximal values are reached on the diagonal.
Note that the Schur product $P\circ \phi$ is a positive semidefinite as well, so it's maximal matrix entries are also assumed on the diagonal. Since $\phi$ is set to $[\phi]_{i,i}=1$ for all $i$ the Schur product does not alter the diagonal entries of $P$ thus $\|P\circ \phi\|_{\max}=\|P\|_{\max}$. 

\end{proof}
\begin{proof}[Proof Claim 4)]
Recall that inequality $(\ref{eq:2normlestuff})$ implies $\|A\|\leq \|A\|_1$ for any symmetric matrix A which yields the first half of claim 4), since $P\circ \phi$ is symmetric. The other half can be obtained by
\begin{equation*}
\|P\circ \phi\|\le\|P\circ \phi\|_1=\max_i\left|\sum_j\phi_{i,j} P_{i,j}\right|\leq \max_i\sum_j|\phi_{i,j}| \|P\|_{\max}= C_\phi  \|P\|_{\max}. 
\end{equation*}
\end{proof}

In this paper, we often concern Riccati type of stochastic equation. In particular, we are often interested in finding bounds for the maximum entry of the solution. To do so, we employ a comparison principle, which generates bounds by comparing with another ODE. In particular, we have the following lemma. 
\begin{lemma}
\label{lem:boundonderivative}
Suppose $X_t=[x_1(t),\ldots, x_n(t)]$ jointly follows an ODE, $\frac{d}{dt}X_t=F(X_t)$. Let $m_t=\max_{1\leq i\leq n'}\{x_i(t)\}$, where $n'$ can be smaller than $n$. Let $i_t$ be the smallest index $i$ such that $x_i(t)=m_t$. Suppose there is a continuous function $g(x,t)$ such that for any $t\geq 0$, 
\[
\frac{d}{dt}x_{i_t}(t)\leq g(x_{i_t}(t),t).
\]

Suppose $y_t$ satisfies $\frac{d}{dt}y_t=g(y_t,t)+\delta_0$ for a fixed $\delta_0>0$ and $y_0>m_0$, then the following hold
\begin{enumerate}[1)]
\item For all $t>0$,  $y_t>m_t$.
\item Suppose $g(x,t)=g(x)=-\frac{c}{\epsilon}x^2+bx+a-\delta_0$, where $a,b,c$ are constants. Let
\[
\Delta_\epsilon:=2\sqrt{\frac{b^2\epsilon^2}{4c^2}+\frac{a\epsilon}{c}}=\Theta(\sqrt{\epsilon}). 
\]
If $y_0>0$, then $y_t\leq \max\{\Delta_\epsilon, y_0\}$ for all $t>0$.   Moreover, when $t>t_*=\frac{c\epsilon}{\Delta_\epsilon}=\Theta(\sqrt{\epsilon})$, $y_t\leq \Delta_\epsilon$.
\item Suppose $z_t$ is a process such that $0<z_t<D$ for all $t>0$,  and $z_t\leq \Delta_\epsilon$ for $t>t_*$, where $\Delta_\epsilon,t_*$ are positive quantities of order $\Theta(\sqrt{\epsilon})$. Suppose  $g(x,t)=\alpha \sqrt{-x z_t}-\frac{\beta}{\epsilon}  x z_t-\gamma-\delta_0$, where $\alpha,\beta,\gamma$ are all positive constants. Then if $y_0<0$, 
\[
y_t\leq -\min \left\{|y_0|, \frac{\gamma^2}{\alpha^2D}, \frac{\gamma \epsilon}{2\beta D} \right\},\quad \forall t>0. 
\] 
Moreover $|y_t|>c_\epsilon$ for all $t>t_*+\frac{3c_\epsilon}{\gamma}=\Theta(\sqrt{\epsilon})$, where
\[
c_\epsilon:=\min\left\{\frac{\gamma^2}{9\alpha^2\Delta_\epsilon},\frac{\gamma \epsilon}{3\beta \Delta_\epsilon} \right\}=\Theta(\sqrt{\epsilon}).
\]
\end{enumerate}
\end{lemma}
\begin{proof}[Proof Claim 1)]
Let $t_1=\inf\{t>0, y_t\leq m_t\}$. By continuity of $m_t$ and $y_t$, $t_1>0$. Suppose $t_1$ is finite, then $y_{t_1}=m_{t_1}$.  Therefore 
\[
\frac{d}{dt}x_{i_{t_1}}(t_1)\leq g(x_{i_{t_1}}(t),t_1)=g(y_{t_1},t_1)=\frac{d}{dt} y(t_1)-\delta_0.
\]
This indicate for sufficiently small $\delta>0$,
\[
x_{i_{t_1}}(t_1-\delta)> x_{i_{t_1}}(t_1)-\delta g(x_{i_{t_1}}(t),t_1)-\frac12  \delta\delta_0>y(t_1)-\delta g(y(t_1),t_1)+\frac12 \delta\delta_0>y(t_{1}-\delta). 
\]
This contradicts with the definition of $t_1$. Therefore $t_1=\infty$. 
\end{proof}
\begin{proof}[Proof Claim 2)]
First we denote the root of $g(x,t)+\delta_0=0$ as
\[
y_{\pm}=-\frac{b\epsilon}{2c}\pm \sqrt{\frac{b^2\epsilon^2}{4c^2}+\frac{a\epsilon}{c}}. 
\]
It is easy to check that $y_+>0>y_-$, while $\Delta_\epsilon=y_+-y_-$. Note that $g(x)+\delta_0 \leq 0$ when $y_t\geq \Delta_\epsilon\geq y_+$. So $y_t$ is decreasing when $y_t$ is above $\Delta_\epsilon$.

Next note that $y_t$ is the solution of a Riccati differential equation. The solution to the Riccati ODE is given by  has the explicit formulation 
\begin{equation}
\label{tmp:Ricc}
\frac{y_t-y_-}{y_t-y_+}=\exp\left(\frac{c}{\epsilon}t (y_+-y_-)\right)\left(\frac{y_0-y_-}{y_0-y_+}\right)
\end{equation}
If $y_0<y_+$, it is easy to check that \eqref{tmp:Ricc} always take negative value, meaning $y_t<y_+<y_+-y_-=\Delta_\epsilon$ for all $t>0$. When $y_0>y_+$ and $t>t_*$, from \eqref{tmp:Ricc} leads to 
\[
\frac{y_t-y_-}{y_t-y_+}\geq \exp\left(\frac{c}{\epsilon}t (y_+-y_-)\right)\geq 2,
\]
so $y_t\leq y_+-y_-=2\sqrt{\frac{b^2\epsilon^2}{4c^2}+\frac{a\epsilon}{c}}=\Theta(\sqrt{\epsilon})$. 
\end{proof}
\begin{proof}[Proof Claim 3)]
Note that $g(x,t)+\delta_0<0$ when $|x|<\min \left\{\frac{\gamma^2}{\alpha^2D}, \frac{\gamma \epsilon}{2\beta D} \right\}$, so $y_t$ will be decreasing if it is above $-\min\left\{\frac{\gamma^2}{\alpha^2D}, \frac{\gamma \epsilon}{2\beta D} \right\}$. This leads to the first part of the claim.

Next note that when $0\geq x\geq -c_\epsilon$ and $t>t_*$, $g(x,t)\leq -\frac\gamma 3$. So $y_t$ will be decreasing with rate at least $\frac{\gamma}{3}$ when $0\geq y_t\geq -c_\epsilon$ and $t>t_*$, this leads to our claim. 

\end{proof}

\subsection{Upper bounds for sample covariance}

\begin{proof}[Proof of Lemma \ref{lem:upperPt}]
Recall that $P_t$ is positive semidefinite, therefore by Lemma \ref{lem:aux1}
\[
\|P_t\|_{\max}=\max\{[P_t]_{i,i},i=1,\ldots,N_x\},
\]
where the components of $P_t$ follows an ODE (\ref{eq:evolutionP_t}). Therefore, in order to  apply Lemma \ref{lem:boundonderivative} Claim 1), it suffices to investigate the ODE deriving the component with the maximal value. Suppose at time $t$, $[P_t]_{k,k}=\|P_t\|_{\max} $ for certain $k$. Considering the time evolution of $[P^L_t]_{k,k}$ given by (\ref{eq:evolutionP_t}), it is given by
\begin{equation}\label{eq:evolutionofPtLidealizedsetting}
\frac{d}{dt} [P_t]_{k,k}= [F_t+F_t^T]_{k,k}+[P_t^\dagger P_t+P_t P_t^\dagger ]_{k,k} -\frac{1}{2\epsilon} [P^L_t \Omega P_t+P_t \Omega P^L_t ]_{k,k}. 
\end{equation}
First note that $[F_t ]_{k,k}=\frac1{M-1}\sum_{i} (x^i_k-\overline{x}_k)(f_k(X^i_t)-\overline f_k)$, where by Assumption \ref{aspt:short} we have
\[
|f_k(X^i)-\overline f_k|\leq \frac{1}{M}\sum_{j=1}^M |f_k(X^i_t)-f_k (X^j_t)|\leq \frac1M \sum_{j=1}^M \sum_{l=1}^{N_x}\mathcal{F}_{\dist(k,l)} |x^i_l-x^j_l|.
\]
This leads to
\begin{align}
\notag
[F_t ]_{k,k}&\leq \frac1{M-1}\sum_{i=1}^M |x^i_k-\overline{x}_k\|f_k(X_t^i)-\overline f_k| \\
\notag
&\leq  \frac1{(M-1)M}\sum_{i,j,l,m} \mathcal{F}_{\dist(k,l)}  |x^i_k-x^{m}_k||x^i_l-x^j_l|\\
\notag
&\leq  \sum_{l=1}^{N_x} \mathcal{F}_{\dist(k,l)}  \sqrt{\frac{\sum_{i,m} |x^i_k-x^{m}_k|^2}{M(M-1)}}\sqrt{\frac{\sum_{i,j}|x^i_l-x^j_l|^2}{M(M-1)}}\\
&\leq  \sum_{l=1}^{N_x} \mathcal{F}_{\dist(k,l)}[P_t]_{k,k}\leq C_f [P_t]_{k,k} \label{eq:ineqC_fP_t}.
\end{align}
Also, note that $[P_t^\dagger P_t]_{k,k}=[P_t^\dagger]_{k,k}[P_t]_{k,k}=1$ due to Defintion \ref{defn:DI}, so 
\begin{equation}\label{eq:secondterm}
[ P_t^\dagger P_t]_{k,k}=1. 
\end{equation}
Lastly, we have 
\begin{equation}
\label{eq:thirdterm}
[P_t^{L}\Omega P_t]_{k,k}=\sum_{i=1}^{N_x}[P_t^{L}]_{k,i}\Omega_{i,i}[P_t]_{i,k}=\sum_{i=1}^{N_x}[P_t]^2_{k,i}\phi_{i,k}\Omega_{i,i}\geq \Omega_{k,k}[P_t]^2_{k,k}\geq \omega_{\min}[P_t]^2_{k,k}. 
\end{equation}
Insert \eqref{eq:ineqC_fP_t},\eqref{eq:secondterm} and \eqref{eq:thirdterm} to \eqref{eq:evolutionofPtLidealizedsetting}, we find
\[
\frac{d}{dt} [P_t]_{k,k}\leq 2C_f [P_t]_{k,k}+2 -\frac{\omega_{\min}}{\epsilon}[P_t]_{k,k}^2. 
\]
Therefore Lemma \ref{lem:boundonderivative} claim 1) applies with 
\[
g(x,t)=2C_f x+2-\frac{\omega_{\min}}{\epsilon}x^2. 
\]
Let $\delta_0=1$, Lemma \ref{lem:boundonderivative} claim 2) yields the result of this lemma.
\end{proof}

\subsection{Lower bounds for sample covariance}

\begin{proof} [Proof of Lemma \ref{lem:lowerPt}]
The proof is similar to the one of Lemma \ref{lem:upperPt}, but we need to change sign, because
\[
-\|P_t\|_{\min}=\max\{-[P_t]_{i,i},i=1,\ldots,N_x\}.
\]
By Lemma \ref{lem:boundonderivative} claim 1), we assume at time $t$,  $\|P_t\|_{\min}=[P_t]_{k,k}$ and investigate the ODE that $-[P_t]_{k,k}$ follows. It is given by 
the inverse of \eqref{eq:evolutionofPtLidealizedsetting}. Following same procedures prior to  (\ref{eq:ineqC_fP_t}), we have
\begin{equation}
\label{tmp:FLPt}
[F_t ]_{k,k}\geq -\sum_{l=1}^{N_x} \mathcal{F}_{\dist(k,l)}  \sqrt{\frac{\sum_{i,m} |x^i_k-x^{m}_k|^2}{M(M-1)}}\sqrt{\frac{\sum_{i,j}|x^i_l-x^j_l|^2}{M(M-1)}}
\geq -C_f \sqrt{[P_t]_{k,k}\|P_t\|_{\max}}
\end{equation}
\eqref{eq:secondterm} remains the same. Finally recall that in \eqref{eq:thirdterm}, we have
\begin{equation}\label{eq:approximationlasttermlowerbound}
[P_t^{L}\Omega P_t]_{k,k}=\sum_{i=1}^{N_x}[P_t]^2_{k,i}\phi_{i,k}\Omega_{i,i}\leq \sum_{i=1}^{N_x} \omega_{\max} \phi_{i,k}[P_t]_{k,k}\|P_t\|_{\max}\leq  \omega_{\max}C_\phi [P_t]_{k,k}\|P_t\|_{\max}. 
\end{equation}
Insert  (\ref{tmp:FLPt}), (\ref{eq:secondterm}), and (\ref{eq:approximationlasttermlowerbound}) into \eqref{eq:evolutionofPtLidealizedsetting}, we find
\[
 -\frac{d}{dt}[P_t]_{k,k}\leq 2C_f \sqrt{-(-[P_t]_{k,k})\|P_t\|_{\max}}+ \frac{\omega_{\max}}{\epsilon}C_\phi \|P_t\|_{\max} [P_t]_{k,k}-2. 
\]
So we can apply Lemma \ref{lem:boundonderivative} claim 3) with $\delta_0=1$ and 
\[
g(x,t)=2C_f \sqrt{-x\|P_t\|_{\max}}+ \frac{\omega_{\max}}{\epsilon}C_\phi x [P_t]_{k,k}-2.
\]
This gives us the claimed result. 
\end{proof}

\section{Proof for filter error analysis in $l_2$ norm }
\subsection{Evolution of component-wise error}
Before we prove the statements of Theorem \ref{theo:accuracyl2} we consider the following auxiliary lemma which will be used several times throughout the remainder of the paper.
\begin{lemma}\label{lem:auxiliarylemma}
Let  $[e_t]_j$ be the $j$-th component of the filter error $e_t$. Then the following holds 
\begin{align}
\label{eq:derrorsquaredbound}
d[e_t]^2_j\leq &\Big(-\alpha_t [e_t]_j^2-2\epsilon^{-1}\sum_{i=1}^j[P_t\circ \tilde{\phi}]_{j,i} [e_t]_i[e_t]_j+\sum_{i\neq j} \calF_{\dist(i,j)}|[e_t]_i|^2+\beta_t\Big)dt+d[\calM_t]_j.
\end{align}
In \eqref{eq:derrorsquaredbound},  $\calM_t$ is a $N_x$ dimensional martingale with components being
\[ 
d[\calM_t]_j=2\sqrt{2}[e_t]_jdW_j -2\epsilon^{-1/2}[e_t]_j[P^L_t \Omega^{1/2} dB]_j. 
\]
In \eqref{eq:derrorsquaredbound}, $\alpha_t$ and $\beta_t$ are two real valued processes given by  
\begin{align*}
\alpha_t&:=2\epsilon^{-1}\rho\|P_t\|_{\min}-C_f-1,\\
\beta_t &:=C^2_f\|P_t\|_{\max}+2 +\epsilon^{-1} C_\phi ^2 \omega_{\max} \|P_t\|^2_{\max}.
\end{align*}
By Lemmas \ref{lem:upperPt} and \ref{lem:lowerPt}, the following holds for all $t\geq 0$
\[
\alpha_t\geq \alpha^*=-C_f-1,
\]
\[
\beta_t\leq \beta^*:=C^2_f\max\{\|P_0\|_{\max},\lambda_{\max}\}+2 +\epsilon^{-1} C_\phi ^2\omega_{\max} \max\{\|P_0\|^2_{\max},\lambda^2_{\max}\}=\Theta(\epsilon^{-1}). 
\]
When $t\geq t_*$, these bounds can be further improved to
\begin{align*}
\alpha_t\geq \alpha_*&:=2\epsilon^{-1}\rho\lambda_{\min}-C_f-1=\mathcal{O}(\epsilon^{-\frac{1}{2}}),\\
\beta_t\leq \beta_* &:=C^2_f\lambda_{\max}+2 +\epsilon^{-1} C_\phi ^2 \omega_{\max} \lambda^2_{\max}=\Theta(1). 
\end{align*}

\end{lemma}

\begin{proof}
Recall the evolution of $X_t$ and $\overline{X}_t$ are given by $dX_t= f(X_t)dt+\sqrt{2}dW_t$ and
\[
d\overline{X}_t=\overline f_tdt- \epsilon^{-1}P^L_t \Omega ( \overline{X}_t dt-dY_t)=\overline f_tdt- \epsilon^{-1}P^L_t \Omega ( \overline{X}_t dt-X_t dt-\sqrt{\epsilon}\Omega^{-1/2}dB_t).
\]
The evolution of the error $e_t=X_t-\overline{X}_t$ is given by the difference between the two, namely
\[
de_t=(f(X_t)-\overline f_t-\epsilon^{-1}P^L_t\Omega e_t) dt+\sqrt{2}dW_t-\epsilon^{-1/2}P_t^L\Omega^{1/2}dB_t. 
\]
The $j$-th component of this differential equation is given by 
\[
d[e_t]_j=(f_j(X_t)-\overline f_j-\epsilon^{-1}[P_t^L \Omega e_t]_j)dt +\sqrt{2}dW_j -\epsilon^{-1/2}[P^L_t \Omega^{1/2} dB]_j,
\]
where $\overline{f}_j$ denotes the j-th component of  $\overline f_t$. 
Ito's formula implies that 
\begin{align}
\notag
d[e_t]^2_j=&\left(2(f_j(X_t)-\overline f_j-\epsilon^{-1}[P_t^L  \Omega e_t]_j)[e_t]_j+2 +\epsilon^{-1}[P^L_t\Omega P^L_t]_{jj}\right)dt\\
\label{tmp:ej}
& +2\sqrt{2} [e_t]_jdW_j -\epsilon^{-1/2}2[e_t]_j[P^L_t \Omega^{1/2}dB]_j . 
\end{align}
To continue, note that
\begin{align}
\notag
|f_j(X_t)-\overline f_j||[e_t]_j|&=\left|f_j(X_t)-\frac{1}{M} \sum_{i=1}^M f_j(X^i_t)\right| \,|[e_t]_j|\\
\notag
&\leq\frac{1}{M} \sum_{i=1}^M|f_j(X_t)-  f_j(X^i_t)\|[e_t]_j|\\
\label{tmp:ej1}
&\leq \frac{1}{M} \sum_{i=1}^M|f_j(\overline{X}_t)-  f_j(X^i_t)\|[e_t]_j|+|f_j(X_t)-  f_j(\overline{X}_t)\|[e_t]_j|.
\end{align}
By Assumption \ref{aspt:short}, the second part of \eqref{tmp:ej1} can be bounded easily by 
\begin{equation}\label{eq:absch2}
|f_j(X_t)-  f_j(\overline{X}_t)\|[e_t]_j|\leq \sum_{i=1}^{N_x} \calF_{\dist(i,j)} |[X_t-\overline{X}_t]_i\|[e_t]_j|=\sum_{i=1}^{N_x} \calF_{\dist(i,j)}|[e_t]_i| |[e_t]_j|\leq \frac12C_f |[e_t]_j|^2+\frac12\sum_{i\neq j}  \calF_{\dist(i,j)}|[e_t]_i|^2. 
\end{equation}
To bound the first part of \eqref{tmp:ej1}, we note by Assumption \ref{aspt:short} and Cauchy Schwarz, 
\begin{align}
\notag
\frac{1}{M}\sum_{i=1}^{M}|f_j(\overline{X}_t)-  f_j(X^i_t)|&\leq \sum_{i,k=1}^{M,N_x} \frac{ \calF_{\dist(k,j)}}{M} |[X^i_t-\overline{X}_t]_k|\\
\notag
&\leq \sqrt{\sum_{k=1}^{N_x} \frac{ \calF_{\dist(k,j)}}{M^2}\sum_{i=1}^{M} |[X^i_t-\overline{X}_t]_k|^2}\\
&\leq\sqrt{\frac{1}{M}C_f^2 (M-1)[P_t]_{k,k}}\leq C_f \|P_t\|^{\frac12}_{\max}\label{tmp:ej2}.
\end{align}
Then multiplication with $|[e_t]_j|$ with \eqref{tmp:ej2} yields
\begin{align*}
\frac{1}{M}\sum_{i=1}^M|f_j(\overline{X}_t)-  f_j(X^i_t)\|[e_t]_j|&\leq \frac12(\frac{1}{M}\sum_{i=1}^M|f_j(\overline{X}_t)-  f_j(X^i_t)|)^2+\frac12|[e_t]_j|^2\\
&\leq  \frac12C^2_f \|P_t\|_{\max}+\frac12|[e_t]_j|^2. 
\end{align*}
Plug these into \eqref{tmp:ej1}, we find
\begin{equation}
\label{tmp:ej3}
|f_j(X_t)-\overline f_j||[e_t]_j|\leq \frac12C^2_f \|P_t\|_{\max}+\frac12|[e_t]_j|^2+\frac12C_f |[e_t]_j|^2+\frac12\sum_{i\neq j}  \calF_{\dist(i,j)}|[e_t]_i|^2.
\end{equation}
Next, we deal with $[P_t^L  \Omega e_t]_j$ in \eqref{tmp:ej}. Define $\tilde{\phi}:=\phi-\rho I$ and obtain the following equality
\begin{equation}\label{eq:absch1}
[P_t^L  \Omega e_t]_j= \sum_{i=1}^{N_x} [P_t\circ \phi]_{j,i} \Omega_{i,i}[e_t]_i=\rho[P_t]_{j,j}\Omega_{j,j} [e_t]_j+\sum_{i=1}^{N_x} [P_t\circ \tilde{\phi}]_{j,i} \Omega_{i,i}[e_t]_i.
\end{equation}
Also note that 
\begin{equation}
\label{eq:oneoftheterms}
[P^L_t\Omega P^L_t]_{j,j}=\sum_i\Omega_{i,i} \phi^2_{i,j} [P_t]^2_{i,j}\leq \|P_t\|_{\max}^2 \omega_{\max} \sum_i \phi^2_{i,j} \leq  C_\phi ^2 \omega_{\max} \|P_t\|_{\max}^2. 
\end{equation}
Plug (\ref{eq:oneoftheterms}), \eqref{eq:absch1}, and (\ref{tmp:ej3}) into \eqref{tmp:ej}, we obtain
\begin{align}
\label{eq:derrorsquaredbound1}
d[e_t]^2_j\leq \Big((1+&C_f-2\epsilon^{-1}\rho \|P_t\|_{\min})[e_t]_j^2-2\epsilon^{-1}\sum_{i=1}^{N_x}[P_t\circ \tilde{\phi}]_{j,i} [e_t]_i[e_t]_j\\
\notag
&+\sum_{i\neq j} \calF_{\dist(i,j)}|[e_t]_i|^2+C_f^2\|P_t\|_{\max}+2+\epsilon^{-1} C_\phi ^2\omega_{\max}  \|P_t\|_{\max}^2\Big)dt\\
\notag
 &\quad+2\sqrt{2}[e_t]_jdW_j -2\epsilon^{-1/2}[e_t]_j[P^L_t \Omega^{1/2} dB]_j. 
\end{align}
\end{proof}

\subsection{Two technical lemmas}
\begin{lemma}[Gr\"onwall's inequality]  
Suppose a real value process $u_t$ satisfies the following for $t\geq t_0$ and constants $\alpha$ and $\beta$:
\begin{equation*}
du_t\le (-\alpha u_t+ \beta)dt +dM_t
\end{equation*}
for some martingale $M_t$. It follows that for any $t\geq t_0$
\begin{equation*}
\E_{t_0}u_t\le u_{t_0}\exp(-\alpha(t-t_0))  +\frac{\beta}{\alpha}(1-\exp(-\alpha (t-t_0))).
\end{equation*}
When $\alpha$ and $\beta$ are both positive, we have further that
\begin{equation*}
\E_{t_0}u_t\le u_{t_0}\exp(-\alpha(t-t_0))  +\frac{\beta}{\alpha}.
\end{equation*}
\end{lemma}

\begin{proof}
Consider $u'_t=\exp(\alpha (t-t_0)) u_t$. Then its evolution follows
\[
du'_t=\alpha u_t dt+\exp(\alpha(t-t_0)) du_t\leq \exp(\alpha(t-t_0)) \beta+ \exp(\alpha(t-t_0)) dM_t.
\]
Integrating both hands from $t_0$ to $t$, then take conditional expectation we have
\[
\E_{t_0} u'_t=u'_{t_0}+\frac{\beta}{\alpha}(\exp(\alpha (t-t_0))-1).
\]
This leads to our claim. 
\end{proof}

\begin{lemma}
\label{lem:integral}
For a positive random variable $X$, if  there are constants $A\geq 2,B\geq 0$ such that $\Prob(X>M)\leq A\exp(-\lambda M)+\exp(\lambda B-\lambda M)$ holds for all $M>0$, then
\[
\E [X]\leq \frac{1+\log 2A}{\lambda}+B.
\]
\end{lemma}
\begin{proof}
Note that if we let $C=\frac1\lambda \log A+B$, which is the point the quantile upper bound takes value $1$, 
\begin{align*}
\E [X]=\int^\infty_0 \Prob(X>x) dx&=\int^\infty_C \Prob(X>x) dx+\int^C_0 \Prob(X>x) dx\\
&\leq \int^\infty_C(A+\exp(\lambda B))\exp(-\lambda x)dx+\int^C_0 1 dx\\
&=\frac{A+\exp(\lambda B)}{\lambda}\exp(-\lambda C)+C=\frac{1+\log (A+\exp(\lambda B))}{\lambda}.
\end{align*}
Finally, since $A\leq A\exp(\lambda B), \exp(\lambda B)\leq A\exp(\lambda B)$, so $ \log (A+\exp(\lambda B)\leq \lambda B+\log 2A $. 
\end{proof}

\subsection{Proof of Theorem \ref{theo:accuracyl2}}
\begin{proof}[Proof Claim 1)]
Note that $\tilde{\phi}\succeq 0$ and thus $\sum_{i,j} [P_t\circ \tilde{\phi}]_{j,i} [e_t]_i[e_t]_j\geq 0$, and $\sum_i \mathcal{F}_{\dist(i,j)}\leq C_f$. So
utilizing Lemma \ref{lem:auxiliarylemma} and summing over all $j$ on both sides of (\ref{eq:derrorsquaredbound}) yields
\begin{equation}
\label{tmp:etl2}
d\|e_t\|^2\leq \left(C_f-\alpha_t \right)\|e_t\|^2 dt +N_x\beta_t dt+d\calM_t',
\end{equation}
where the martingale is given by  
\begin{align*}
d\mathcal{M}'_t=\sum_{j=1}^{N_x}2\sqrt{2}[e_t]_jdW_j -2r^{-1/2}[e_t]_j[P^L_s \Omega^{1/2} dB]_j=2\sqrt{2} e_t^TdW_t-2r^{-1/2} e_t^TP^L_t  \Omega^{1/2} dB_t.
\end{align*}
For $t\in [0,t_*]$, \eqref{tmp:etl2} can be further upper-bounded by 
\[
d\|e_t\|^2\leq \left(C_f-\alpha^*\right)\|e_t\|^2 dt +N_x\beta^* dt+d\calM_t'.
\]
Employing Gronwall's inequality, there is a constant $D$ such that 
\[
\E \|e_{t_*}\|^2\leq \exp((2C_f+1) t_*)\E \|e_0\|^2+N_x\beta^* \frac{\exp((2C_f+1)t_*)-1}{2C_f+1}=\Theta(\epsilon^{-1}).
\]
For $t\geq t_*$,  \eqref{tmp:etl2} can be further upper-bounded by 
\[
d\|e_t\|^2\leq -\alpha'_*\|e_t\|^2 dt +N_x\beta_* dt+d\calM_t'.
\]
Employing Gronwall's inequality, we find that with $\alpha_*'=\alpha_*-C_f$, 
\begin{equation}
\mathbb{E}\|e_t\|^2\le \E\|e_{t_*}\|^2 \exp(-\alpha_*'(t- t_*)) +\frac{N_x\beta_*}{\alpha_*'}(  1-\exp(-\alpha_*' (t-t_*))). 
\end{equation}
Since $\alpha_*'=\alpha_*-C_f=\Theta(\epsilon^{-1/2})$, $\beta_*= \mathcal{O}(1)$, and $\epsilon^{-1}\exp(-\lambda \epsilon^{-1/2})=o(1)$ for any $\lambda>0$,  so we have proved for claim 1). \\
\end{proof}
\begin{proof}[Proof Claim 2).]
First we note the  quadratic variation of the martingale term $\calM_t'$ is given by 
\[
\frac{d}{dt}\langle\calM'\rangle_t=8 \|e_t\|^2+4\epsilon^{-1} \|\Omega^{1/2}P^L_te_t\|^2\leq (8+4\epsilon^{-1}\omega_{\max} \|P^L_t\|^2) \|e_t\|^2.
\]
So by Ito's formula on $\exp(\lambda \|e_t\|^2)$, the following holds with $\alpha_t'=\alpha_t-C_f$, 
\begin{align*}
d\exp(&\lambda\|e_t\|^2)\leq ((-\lambda \alpha'_t\|e_t\|^2 +\lambda\beta_t N_x)dt+\lambda d\calM_t') \exp(\lambda \|e_t\|^2)+\frac12 \lambda^2\exp(\lambda\|e_t\|^2) d\langle \calM\rangle_t dt\\
&\leq (-\gamma_t\|e_t\|^2 +\lambda\beta_t N_x) \exp(\lambda \|e_t\|^2)dt+ \lambda\exp(\lambda\|e_t\|^2) d\calM_t'. 
\end{align*}
where
\[
\gamma_t=\lambda \alpha'_t-4\lambda^2-2\lambda^2\omega_{\max} \epsilon^{-1}\|P^L_t\|^2. 
\]
By Lemma \ref{lem:upperPt} and \ref{lem:lowerPt}, we have for all $t>0$ 
\[
-\gamma_t\leq \gamma^*=\lambda (2C_f+1) +4 \lambda^2+2\lambda^2 \omega_{\max} C_{\phi}^2 \max\{\|P_0\|^2_{\max}, \lambda_{\max}^2 \},
\]
and for $t\geq t_*$
\[
\lambda\leq \lambda_*=\frac{\alpha'_*}{8+4 \omega_{\max} C_{\phi}^2  \lambda_{\max}^2}=\Theta(\epsilon^{-1/2}). 
\]
\[
\gamma_t\geq \lambda \alpha'_* -4 \lambda^2-2\lambda^2 \omega_{\max} C_{\phi}^2  \lambda_{\max}^2\geq \frac12\lambda \alpha'_*. 
\]
For $t\leq t_*$,  by Gronwall's inequality  we have
\begin{equation}
\label{tmp:laptep}
\E \exp(\lambda \|e_{t_*}\|^2)\leq \exp((\gamma^*+\lambda\beta^* N_x)t_*)\exp(\lambda \|e_0\|^2).
\end{equation}
And when $t\geq t_*$, 
\begin{equation}\label{eq:hildabschaetzung2}
d\exp(\lambda\|e_t\|^2)\leq (\lambda\beta_* N_x -\frac12\lambda\alpha'_* \|e_t\|^2) \exp(\lambda \|e_t\|^2)dt+\lambda\exp(\lambda\|e_t\|^2) d \calM_t'.
\end{equation}
Note that when $\frac14 \lambda \alpha'_* \|e_t\|^2\leq \lambda \beta_* N_x$,
\[\exp(\lambda  \|e_t\|^2)\leq \exp\left(\frac{4\lambda N_x \beta_*}{\alpha'_*}\right),\]
we obtain
\[
(\lambda\beta_* N_x -\tfrac12\lambda\alpha'_* \|e_t\|^2) \exp(\lambda \|e_t\|^2)\leq - \lambda \beta_* N_x \exp(\lambda \|e_t\|^2)+2\lambda\beta_* N_x\exp\left(\frac{4\lambda N_x \beta_*}{\alpha'_*}\right).
\]
Otherwise, when $\frac14 \lambda \alpha'_* \|e_t\|^2\geq \lambda \beta_* N_x$, we have
\[
(\lambda\beta_* N_x -\tfrac12\lambda\alpha'_* \|e_t\|^2) \exp(\lambda \|e_t\|^2)\leq -\frac14\lambda\alpha'_* \|e_t\|^2 \exp(\lambda \|e_t\|^2)\leq -\lambda \beta_* N_x \exp(\lambda \|e_t\|^2).  
\]
In summary, we always have
\begin{equation}\label{eq:hilfsabschaetzung}
(\lambda\beta_* N_x -\tfrac12\lambda\alpha'_* \|e_t\|^2) \exp(\lambda \|e_t\|^2)\leq -\lambda \beta_* N_x \exp(\lambda \|e_t\|^2)+2\lambda\beta_* N_x\exp\left(\frac{4\lambda  N_x \beta_*}{\alpha'_*}\right).
\end{equation}
Inserting (\ref{eq:hilfsabschaetzung}) in (\ref{eq:hildabschaetzung2})  yields 
\[
d\exp(\lambda\|e_t\|^2)\leq \Big[-\lambda \beta_* N_x \exp(\lambda \|e_t\|^2)+2\lambda\beta_* N_x\exp\left(\frac{4\lambda  N_x \beta_*}{\alpha'_*}\right)\Big]dt+ \lambda\exp(\lambda\|e_t\|^2) d\calM_t'.
\]
After applying Gr\"onwall's inequality and \eqref{tmp:laptep} we obtain the following 
\begin{align*}
\mathbb{E}[\exp(\lambda\|e_t\|^2)]&\leq \exp{(-\lambda \beta_* N_x (t-t_*))} \mathbb{E}[\exp(\lambda \|e_{t_*}\|^2)]+2\exp\left(\frac{4\lambda  N_x \beta_*}{\alpha_*}\right)\\
&\leq \exp(-(\gamma^*+\lambda\beta^* N_x)t_*-\lambda \beta_* N_x (t-t_*))\mathbb{E}[\exp(\lambda \|e_{0}\|^2)]+2\exp\left(\frac{4\lambda  N_x \beta_*}{\alpha'_*}\right).
\end{align*}
When $t\to \infty$, this leads to claim 2):
\[
\limsup_{t\to \infty}\E \exp(\lambda\|e_t\|^2)\leq  2\exp\left(\frac{4\lambda N_x \beta_*}{\alpha_*'}\right). 
\]
\end{proof}
\begin{proof}[Proof Claim 3)]
We consider function 
\[
g(x)=(\lambda\beta_* N_x -\frac12\lambda\alpha'_* x) \exp(\lambda x)
\]
By finding the critical point, it is easy to see 
\[
g(x)\leq g\Big(\frac{2\beta_* N_x}{\alpha'_*}-\frac{1}{\lambda}\Big)=\frac{\alpha'_*}{2 \mathbb{e}}\exp\Big(\frac{2\lambda_* \beta N_x}{\alpha'_*}\Big)=:G_*=\Theta(\epsilon^{-1/2}) 
\]
Combine this with \eqref{eq:hildabschaetzung2}, we find
\[
d\exp(\lambda\|e_t\|^2)\leq G_*dt+\lambda \exp(\lambda\|e_t\|^2) d \calM_t',\quad \forall t\geq t_*.
\]
So by Dynkin's formula, if we let $\tau=\min\{t:t\geq t_0, \|e_t\|^2\geq M\}$, then 
\[
\E_{t_0} \exp(\lambda\|e_{T\wedge \tau}\|^2)\leq  \exp(\lambda\|e_{t_0}\|^2) +\E \int^{T\wedge \tau}_{t_0} G_* dt\leq \exp(\lambda\|e_{t_0}\|^2) + G_*T.
\]
By Markov inequality we have
\begin{align*}
\Prob(\max_{t_0\leq t\leq T} \|e_t\|^2\geq M)= \Prob_{t_0}&(\|e_{T\wedge \tau}\|^2\geq M)\leq \frac{\E_{t_0} \exp(\lambda\|e_{T\wedge \tau}\|^2)}{\exp(\lambda M)}\\
&\leq \frac{\alpha_*' T}{ 2\mathbb{e}}\exp\left(\frac{2\lambda \beta_* N_x}{\alpha_*'}-\lambda M\right)+\exp( \lambda\|e_{t_0}\|^2   -\lambda M).
\end{align*}
Then by Lemma \ref{lem:integral}, 
\[
\E_{t_0} \max_{t_0\leq t\leq T} \|e_{T}\|^2\leq \frac1\lambda+\frac{2\beta_* N_x}{\alpha_*'}+\frac1\lambda\log\left(\frac{\alpha'_* T}{ \mathbb{e} }\right)+\frac1\lambda+\|e_{t_0}\|. 
\]
We take $\lambda=\lambda_*=\Theta(\epsilon^{-\frac12})$ to obtain our claimed result. 
\end{proof}

\section{Proof for component-wise filter error analysis}

\subsection{Component-wise Lyapunov weights}
In order to bound $[e_t]^2_i$ in long time, it is necessary to build a Lyapunov function for it. The main challenge here is that dynamics of $[e_t]^2_i$ is coupled with the error of other components. The idea is here to find a weight vector $v^i$ so that $E^i_t=\sum_j v^i_j [e_t]^2_j$ is a Lyapunov function. The design of $v^i$ happens to relate to the structure of $\phi$, and can be expressed as the Green function of a Markov chain. 
\begin{lemma}
\label{lem:vexist}
Under Assumption \ref{aspt:diagd}. Let $T$ be a random variable of geometric-$q$ distribution, that is 
\[
\Prob(T=n)=(1-q)q^{n-1},\quad n=1,2,\ldots.
\]
Consider a Markov chain $X_t$ on the points $\{1,\ldots, N_x\}$. Its transition probability is given by 
\[
\Prob(X_{t+1}=j|X_t=i)=\begin{cases}\frac1q\phi_{i,j}\quad &j\neq i\\
1-\frac1q\sum_{j\neq i}\phi_{i,j}\quad &j=i.
\end{cases}
\]
Fix an index $i\in \{1,\ldots, N_x\}$. Define vector $v^{i}$, where its components are given by 
\[
v^{i}_j=\E\left(\sum^T_{k=1} \mathbf{1}_{X_k=i}\bigg|X_1=j\right). 
\]
Then $v^{i}$ satisfies the following properties
\begin{enumerate}[1)]
\item $v^{i}_j\geq 0, \forall j$ and in specific $v^i_i\geq 1-q$.
\item For all index $j$, $\sum_{l\neq j} \phi_{j,l}v^i_l\leq v^i_j$.
\item $\sum_{j=1}^{N_x} v^i_j\leq 1$. 
 \end{enumerate}
\end{lemma}
\begin{proof}[Proof Claim 1)]
Since $\sum^T_{k=1} \mathbf{1}_{X_k=i}\geq 0$ a.s., so $v^{i}_j\geq 0$. Moreover, 
\[
v^i_i=\E\left(\sum^T_{k=1} \mathbf{1}_{X_k=i}\bigg|X_1=i\right)
\geq \E\left(\mathbf{1}_{T=1, X_1=i}\bigg|X_1=i\right)=1-q. 
\]
\end{proof}
\begin{proof}[Proof Claim 2:]
Next, by doing a first step analysis of Markov chain, we find that 
\begin{equation}
\label{tmp:vij}
v^i_j =(1-q)\cdot \mathbf{1}_{j=i}+q \left(1-\frac1q\sum_{l\neq j}\phi_{j,l}\right)v^i_j+q\cdot \frac1q\sum_{l\neq j}\phi_{j,l}v^i_l.
\end{equation}
Since $\sum_{l\neq j}\phi_{j,l}\leq q<1$, we have
\[
v^i_j\geq \sum_{l\neq j}\phi_{l,j}v^i_l.
\]

\end{proof}
 
\begin{proof}[Proof Claim 3)]
We sum \eqref{tmp:vij} over all $j$ and obtain
\begin{align*}
\sum_{j=1}^{N_x}v^i_j &=(1-q)+q\sum_{j=1}^{N_x}\left(1-\frac1q\sum_{l\neq j}\phi_{j,l}\right)v^i_j+\sum_{j=1}^{N_x}\sum_{l\neq j}\phi_{j,l}v^i_l\\
&\leq 1-q+\sum_{j=1}^{N_x}\sum_{l\neq j}\phi_{j,l}v^i_l=1-q+ \sum_{l=1}^{N_x} v_l^i\left(\sum_{j\neq l }\phi_{j,l}\right).
\end{align*}
Therefore we have 
\[
(1-q)\sum_{j=1}^{N_x}v^i_j \leq \sum_{j=1}^{N_x}(1-\sum_{j\neq l }\phi_{j,l})v^i_j\leq 1-q,
\]
which leads to our claim. 
\end{proof}

\subsection{Proof of Theorem \ref{the:individualcomp} }

\begin{proof}[Proof Claim 1)] Recall that Lemma \ref{lem:auxiliarylemma} has shown that 
\begin{align}
\label{tmp:ei}
d[e_t]^2_i\leq &\Big(-\alpha_t [e_t]_i^2-2\epsilon^{-1}\sum_{j=1}[P_t\circ \tilde{\phi}]_{i,j} [e_t]_i[e_t]_j+\sum_{j\neq i} \calF_{\dist(i,j)}|[e_t]_j|^2+\beta_t\Big)dt+d[\calM_t]_i. 
 \end{align}
Recall that $\tilde{\phi}=\phi-\rho I$. 
In the following, we use $P_{j,i}$ to denote the $(j,i)$-th component of $P_t$.  Then by Cauchy Schwartz and Young's inequality
\begin{align*}
-2 [P_t\circ \tilde{\phi}]_{i,j} [e_t]_i[e_t]_j= -2\phi_{j,i} P_{j,i}[e_t]_i[e_t]_j&\leq -2\phi_{j,i} \sqrt{P_{j,j}}[e_t]_j\sqrt{P_{i,i}}[e_t]_i \\
&\leq \phi_{i,j}(P_{j,j}[e_t]_j^2 +P_{i,i}[e_t]_i^2),\quad \text{ for }j\neq i.
\end{align*}
Then note that
\[
-2P_{i,i}\phi_{i,i} [e_t]_i^2+\sum_{i\neq j} \phi_{i,j}P_{i,i}[e_t]_i^2\leq (q-2) P_{i,i}[e_t]_i^2<-P_{i,i}[e_t]_i^2,
\] 
so \eqref{tmp:ei} leads to 
\begin{equation}
\label{tmp:ei1}
d[e_t]^2_i\leq \left(\sum_{j\neq i} (\mathcal{F}_{\dist(i,j)}+\epsilon^{-1}\phi_{i,j}P_{j,j} )[e_t]^2_j-\alpha_t[e_t]_i^2-\epsilon^{-1}P_{i,i} [e_t]_i^2 +\beta_t\right)dt+d[\calM_t]_i. \\
\end{equation}
We denote the vector $E_t=[e_1^2,e_2^2,\cdots e_N^2]^T$. Further we define vector $v^i$, of which the component is given by Lemma \ref{lem:vexist}. Denote
\[
E^i_t=\langle v^i, E_t\rangle,\quad \calM^i_t=\langle v^i,\calM_t\rangle.
\]
Then the SDE of $E^i_t$ can be bounded by a linear combination of \eqref{tmp:ei1}, which is 
\begin{align}
\notag
d E^i_t&\leq \sum_{j=1}^{N_x}  \left(-\alpha_t v^i_j [e_t]_j^2-\epsilon^{-1}  v^i_j(P_{j,j}[e_t]_j^2-\sum_{l\neq j}\phi_{j,l} P_{l,l}  e^2_l)+\sum_{l\neq j}v^i_j \mathcal{F}_{\dist(j,l)}e_l^2\right)+\beta_t +d\mathcal{M}^i_t\\
\notag
&= \sum_{j=1}^{N_x}  \left(-\alpha_t v^i_j [e_t]_j^2-\epsilon^{-1}  (v^i_j P_{j,j}[e_t]_j^2-\sum_{l\neq j}v^i_l\phi_{l,j} P_{j,j}  [e_t]^2_j)+\sum_{l\neq j} \mathcal{F}_{\dist(j,l)}v^i_l [e_t]_j^2\right)+\beta_t +d\mathcal{M}^i_t\\
\label{tmo:Eit1}
&\leq \sum_{j=1}^{N_x}  \left(-\alpha_t v^i_j [e_t]_j^2 v^i_j +C_{\mathcal{F}}\phi_{j,l}v^i_l [e_t]_j^2\right)+\beta_t +d\mathcal{M}^i_t.\\
\label{tmo:Eit2}
&\leq \sum_{j=1}^{N_x}  (-\alpha_t+C_{\mathcal{F}}) v^i_j [e_t]_j^2 +\beta_t +d\mathcal{M}^i_t=(-\alpha_t+C_{\mathcal{F}}) E^i_t +\beta_t +d\mathcal{M}^i_t. 
\end{align}
We have used claim 3) and 2) of Lemma \ref{lem:vexist} at \eqref{tmo:Eit1} and \eqref{tmo:Eit2}. 

Between time $0$ and $t_\epsilon$, recall the upper bound in Lemma \ref{lem:auxiliarylemma}, apply Gronwall's inequality
\[
\E E^i_{t_\epsilon}\leq \exp ((C_\mathcal{F}-\alpha^*)t_\epsilon)\left(\E E^i_0+ \frac{\beta^*}{C_\mathcal{F}-\alpha^*} \right). 
\]
Then after $t_\epsilon$, for any $t$, apply Gronwall's inequality
\begin{align*}
\E E^i_{t}&\leq \exp ((C_\mathcal{F}-\alpha_*)t_\epsilon)\E E^i_{t_\epsilon}+\frac{\beta_*}{\alpha_*-C_\mathcal{F}} \\
&\leq \exp (C_\mathcal{F}t-\alpha^*t_\epsilon-\alpha_*(t-t_\epsilon))\left(\E E^i_0+ \frac{\beta^*}{C_\mathcal{F}-\alpha^*} \right)+\frac{\beta_*}{\alpha_*-C_\mathcal{F}}. 
\end{align*}
Recall that in Lemma \ref{lem:auxiliarylemma}, $\alpha_*=\Theta(\epsilon^{-1/2}),\beta^*=\Theta(\epsilon^{-1}),\alpha^*=\beta_*=\Theta(1)$. So if $t>t_0$, for certain constants $c$ and $C$
\[
-(C_\mathcal{F}t-\alpha^*t_\epsilon-\alpha_*(t-t_\epsilon))\geq c \epsilon^{-1/2},\quad  \E E^i_0\leq \max_i \{|[e_t]_i(0)|^2\}\sum_j  v^i_j\leq C, 
\]
\[
\frac{\beta^*}{C_{\mathcal{F}}-\alpha^*}\leq C\epsilon^{-1},\quad \frac{\beta_*}{\alpha_*-C_\mathcal{F}}\leq C\epsilon^{1/2}. 
\]
Therefore when $\epsilon$ is small enough, $\E E^i_{t}\leq 2C\sqrt{\epsilon}$, which is our claim 1). 
\end{proof}

\begin{proof}[Proof Claim 2)] First recall  the individual martingale driving $E^i_t$ is  given by 
\begin{align*}
d\calM^i_t&=\sum_j v^i_j\sqrt{8}[e_t]_jdW_j -2v^i_j\epsilon^{-1/2}[e_t]_j[P^L_t \Omega^{1/2} dB]_j.
\end{align*}
The corresponding quadratic variation is bounded by 
\begin{align*}
\frac{d}{dt}\langle\calM^i\rangle_t&=8 \sum_{j=1}^{N_x} [e_t]_j^2 (v^i_j)^2+4\epsilon^{-1}\sum_{j=1}^{N_x} (v_j^i)^2 [e_t]_j^2 \sum_{l=1}^{N_x}[P^L_t]^2_{j,l}[\Omega]_{l,l}\\
&\leq 8 \sum_{j=1}^{N_x} [e_t]_j^2 v^i_j+4\omega_{\max}\epsilon^{-1}\|P_t\|_{\max}^2 \sum_{j=1}^{N_x} v_j^i [e_t]_j^2 \leq 4\beta_t E^i_t. 
\end{align*}
Denote $\alpha'_t=\alpha_t-C_{\mathcal{F}}$, (which is slightly different from the one in the proof of Theorem \ref{theo:accuracyl2}) then recall from \eqref{tmo:Eit2} we have
\[
dE^i_t\leq -\alpha'_t E^i_t dt+\beta_tdt+d\mathcal{M}_t^i.
\]
By Ito's formula on $\exp(\lambda E_t^i)$, we have
\begin{align}
\notag
d\exp(\lambda E_t^i)&\leq (-\frac12\lambda \alpha'_t E^i_t +4\lambda\beta_t )dt+\lambda d\calM^i_t) \exp(\lambda E^i_t)+\frac12 \lambda^2\exp(\lambda E^i_t) d\langle \calM^i\rangle_t\\
\label{tmp:lEt1}
&\leq (-\frac12(\lambda \alpha'_t-4\lambda^2\beta_t ) E^i_t +\lambda\beta_t  ) \exp(\lambda E^i_t)dt+ \lambda\exp(\lambda E^i_t) d\calM^i_t. 
\end{align}
From time $0$ to $t_\epsilon$, by Lemma \ref{lem:auxiliarylemma}, 
\[
\alpha_t'=\alpha_t-C_{\mathcal{F}}\geq \alpha^{*}-C_{\mathcal{F}},\quad \beta_t\leq \beta^*,
\]
by Gronwall's inequality, for all $i$
\begin{equation}
\label{tmp:lEt4}
\E \exp(\lambda E_{t_\epsilon}^i)\leq \exp (t_\epsilon (-\tfrac12\lambda (\alpha^*-C_{\mathcal{F}})+2\lambda^2\beta^*+\tfrac12 \lambda \beta^*)) \exp(\lambda \max_{i}\{|e^i_t(0)|^2\})<\infty. 
\end{equation}
When $t>t_\epsilon$, Lemma \ref{lem:auxiliarylemma} further shows that 
\[
\alpha_t'=\alpha_t-C_{\mathcal{F}}\geq \alpha_*':=\alpha_*-C_{\mathcal{F}},\quad \beta_t\leq \beta_*.
\]
Consider $\lambda\leq \lambda_*=\frac{\alpha'_*}{8\beta_*}$, then
\[
-\frac12(\lambda \alpha_*-4\lambda^2\beta_*) = -\frac14 \lambda\alpha_*. 
\]
Then for $t>t_\epsilon$ and $\lambda<\lambda_\epsilon$, we have the following upper bound from \eqref{tmp:lEt1}
\begin{equation}
\label{tmp:lEt2}
d\exp(\lambda E_t^i) \leq (-\frac14 \lambda\alpha'_* E^i_t +\lambda\beta_*  ) \exp(\lambda E^i_t)dt+ \lambda\exp(\lambda E^i_t) d\calM^i_t. 
\end{equation}
When $\epsilon$ is small enough, $\alpha'_*>0$. Then if $\frac18 \lambda\alpha'_* E^i_t \leq \lambda\beta_*$, 
\[
(-\frac14 \lambda\alpha'_* E^i_t +\lambda\beta_*  ) \exp(\lambda E^i_t)+\frac18 \lambda \alpha'_*\exp(\lambda E^i_t) \leq 2\lambda\beta_*\exp(8\lambda\beta_*/\alpha'_*). 
\]
If $\frac18 \lambda\alpha'_* E^i_t \geq \lambda\beta_*$, 
\[
(-\frac14 \lambda\alpha'_* E^i_t +\lambda\beta_*  ) \exp(\lambda E^i_t)\leq -\frac18 \lambda \alpha'_*\exp(\lambda E^i_t). 
\]
In summary, we always have
\[
(-\frac14 \lambda\alpha'_* E^i_t +\lambda\beta_*  ) \exp(\lambda E^i_t)\leq 
-\frac18 \lambda\alpha'_*\exp(\lambda E^i_t)+2\lambda\beta_*\exp(8\lambda\beta_*/\alpha'_*).
\]
Plug this into \eqref{tmp:lEt2}, we have
\begin{equation}
\label{tmp:lEt3}
d\exp(\lambda E_t^i) \leq (-\frac18 \lambda\alpha'_*\exp(\lambda E^i_t)+2\lambda\beta_*\exp(8\lambda\beta_*/\alpha'_*))dt+ \lambda\exp(\lambda E^i_t) d\calM^i_t. 
\end{equation}
So Gronwall's inequality and implies for $t\geq t_\epsilon$
\begin{align*}
\E \exp(\lambda E_t^i) \leq \exp (-\tfrac18 \lambda \alpha'_* (t-t_\epsilon)) \E \exp(\lambda E_{t_\epsilon}^i) +16 \frac{\beta_*}{\alpha'_*} \exp(8\lambda\beta_*/\alpha'_*)
\end{align*}
The first term on the right converges to zero as $t\to \infty$ because of bound \eqref{tmp:lEt4}. We have our claim 2) because of $\beta_*=\Theta(1)$, $\alpha'_*=\Theta(\epsilon^{-1/2})$, moreover $E_t^i\geq v^i_i[e_t]^2_i\geq (1-q) [e_t]^2_i$ by Lemma \ref{lem:vexist} claim 1). 

\end{proof}
\begin{proof}[Proof Claim 3)] We  consider function 
\[
g(x)=(-\tfrac14 \lambda\alpha'_* x +\lambda\beta_*  ) \exp(\lambda x)
\]
and by finding the critical point, it is easy to see 
\[
g(x)\leq g\left(\frac{4\beta_*}{\alpha'_*}-\frac1\lambda\right)=\frac{\alpha_*}{4\mathbb{e}}\exp\left(\frac{4\lambda \beta_*}{\alpha_*'}\right)=:G_*. 
\]
Plug this into \eqref{tmp:lEt2}, we find for all $t>t_0$, 
\[
d\exp(\lambda E^i_t)\leq G_*dt+\lambda \exp(\lambda E^i_t) d \calM^i_t.
\]
So by Dynkin's formula, if we let $\tau_i=\min\{t: E^i_t\geq M\}$, then 
\[
\E_{t_0} [\exp(\lambda E^i_{T\wedge \tau})]\leq \exp(\lambda E^i_{t_0}) +\E_{t_0} [\int^{T\wedge \tau}_{t_0} G_* dt]\leq\exp(\lambda E^i_{t_0})  + G_*T. 
\]
Recall that $E^i_t\geq v^i_i[e_t]^2_i\geq (1-q) [e_t]^2_i$. 
By Markov inequality
\begin{align*}
\Prob\left(\sup_{t_0\leq t\leq T}  \{[e_t]^2_i\}\geq  \frac{M}{1-q}\right)
&\leq  \Prob\left(\sup_{t_0\leq t\leq T}  E^i_t\geq  M\right)\\
&\leq \frac{\E \exp(\lambda E^i_{T\wedge \tau})}{\exp(\lambda M)}\\
&\leq  \frac{\alpha_* T}{4\mathbb{e}} \exp \left( \frac{4\lambda \beta_*}{\alpha_*'}-((1-q)\lambda)\frac{M}{(1-q)}\right)+\exp \left(\lambda E^i_{t_0}-((1-q)\lambda)\frac{M}{(1-q)}\right). 
\end{align*}
Note that $E^i_{t_0}=\sum_{j=1}^{N_x} v^i_j [e_{t_0}]_j^2\leq \max_j [e_{t_0}]_j^2$.Then by Lemma \ref{lem:integral}, we have
\[
\E[ \sup_{t_0\leq t\leq T}  \{[e_t]^2_i\}]\leq \frac1{(1-q)\lambda}+\frac{4\beta_* }{\alpha'_*(1-q)}+\frac{1}{\lambda}\log\left(\frac{\alpha'_* T}{2 e}\right)+\max_i [e_{t_0}]_j^2. 
\]
We have claim 3) because $\beta_*=\Theta(1)$, $\alpha'_*=\Theta(\epsilon^{-1/2})$ and taking $\lambda=\lambda_\epsilon=\Theta(\epsilon^{-1/2})$. 
\end{proof}
\begin{proof}[Proof Claim 4)] We note 
\begin{align*}
&\Prob\left(\max_{i}\sup_{t_0\leq t\leq T}  \{[e_t]^2_i\}\geq  \frac{M}{1-q}\right)\\
&\leq \sum_{i=1}^{N_x} \Prob\left(\sup_{t_0\leq t\leq T}  \{[e_t]^2_i\}\geq  \frac{M}{1-q}\right)\\
&\leq \sum_{i=1}^{N_x} \Prob\left(\sup_{t_0\leq t\leq T}  E^i_t\geq  M\right)\\
&\leq \sum_{i=1}^{N_x}\frac{\E \exp(\lambda E^i_{T\wedge \tau})}{\exp(\lambda M)}\\
&\leq  \frac{N_x\alpha_* T}{4\mathbb{e}} \exp \left( \frac{4\lambda \beta_*}{\alpha_*'}-((1-q)\lambda)\frac{M}{(1-q)}\right)+N_x\exp \left(\lambda E^i_{t_0}-((1-q)\lambda)\frac{M}{(1-q)}\right). 
\end{align*}
Then by Lemma \ref{lem:integral} and $E^i_{t_0}\leq \sum_{j=1}^{N_x} v^i_j [e_{t_0}]_j^2\leq \max_j \{[e_{t_0}]_j^2\}$ 
\[
\E_{t_0} \max_{i}\sup_{t_0\leq t\leq T}  \{[e_t]^2_i\}\leq \frac1{(1-q)\lambda}+\frac{4\beta_* }{\alpha_*(1-q)}+\frac{1}{\lambda}\log\left(\frac{\alpha_* N_x T}{ \mathbb{e}}\right)+\log N_x+ \max_j \{[e_{t_0}]_j^2\}. 
\]
We take $\lambda=\lambda_\epsilon=\Theta(\epsilon^{-\frac12})$ to obtain our claimed result. 
\end{proof}

\bibliographystyle{plain}
\bibliography{survey_paper}
\end{document}